\documentclass[11pt,reqno]{amsart}

\usepackage[utf8]{inputenc}
\usepackage[T1]{fontenc}

\usepackage{amsmath, amssymb, amsthm}
\usepackage{geometry}
\usepackage{enumitem}
\usepackage{mathrsfs}
\usepackage{hyperref}
\hypersetup{
  colorlinks,
  citecolor=blue,
  linkcolor=blue,
  urlcolor=blue}
\usepackage{marginnote}
\usepackage{xcolor}
\usepackage{upgreek}
\usepackage{tikz}
\usepackage{tikz-cd}
\usetikzlibrary{calc}
\usepgflibrary {shadings} 
\usetikzlibrary{intersections}
\usetikzlibrary{patterns}
\usepackage{mathtools}
\usepackage{verbatim}
\usepackage{lipsum}  
\usepackage{stmaryrd}
\usepackage{esint}

\numberwithin{equation}{section}

\theoremstyle{plain}
\newtheorem{theorem}{Theorem}
\newtheorem{lemma}{Lemma}[section]
\newtheorem{proposition}{Proposition}[section]

\newtheorem{corollary}{Corollary}[section]

\theoremstyle{definition}
\newtheorem{definition}{Definition}

\newtheorem{remark}{Remark}

\newcommand{\Z}{\mathbb{Z}}
\newcommand{\C}{\mathbb{C}}
\newcommand{\R}{\mathbb{R}}

\renewcommand{\O}{\mathcal{O}}

\newcommand{\s}{\mathbb{S}}

\newcommand{\Sing}{\operatorname{Sing}}

\newcommand{\Ric}{\operatorname{Ric}}

\renewcommand{\d}{\mathrm{d}}
\newcommand{\llb}{\llbracket}
\newcommand{\rrb}{\rrbracket}

\newcommand{\eps}{\varepsilon}

\begin{document}

\title[Harmonic maps to $\s^1$ with higher dimensional singular set]{Harmonic maps to the circle with higher dimensional singular set}
	
    \author[M. Badran]{Marco Badran}
	\address{ETH Z\"urich, Department of Mathematics, Rämistrasse 101, 8092 Zürich, Switzerland.}
 	\email{marco.badran@math.ethz.ch}	
 	
 	\begin{abstract}
 		In a closed, oriented ambient manifold $(M^n,g)$ we consider the problem of finding $\mathbb{S}^1$-valued harmonic maps with prescribed singular set. We show that the boundary of any oriented $(n-1)$-submanifold can be realised as the singular set of an $\s^1$-valued map, which is classically harmonic away from the singularity and distributionally harmonic across. If the singular set $\Gamma$ is also embedded and $C^{1,1}$, we consider three variational relaxations of the same problem and show that the energy of minimisers converges, after renormalisation, to the volume $\mathcal{H}^{n-2}(\Gamma)$ plus a lower-order ``renormalised energy'' --- common to all relaxations --- describing an energetic interaction between different components of the singular set. 
 	\end{abstract}
    
    \maketitle

\section{Introduction}

\subsection{Singular harmonic maps and renormalised energies}
Singular harmonic maps to the circle were first studied in the pioneering work of Bethuel, Brezis and H\'elein  \cite{Bethuel-Brezis-Helein1994} in relation with Ginzburg--Landau vortices. Even though the class of $W^{1,2}(\Omega,\s^1)$ maps with trace $\varphi\colon \partial\Omega\to\s^1$ is empty whenever $\deg(\varphi)\ne 0$, they showed that minimisers $(u_\epsilon)$ of the Ginzburg--Landau energy
\begin{equation}\label{eq: GL}
	E_\epsilon(u)=\int_\Omega |\d u|^2+\frac1{\epsilon^{2}}(1-|u|^2)^2,
\end{equation}
on a simply connected planar domain $\Omega$, converge in the singular limit to a map $u\colon \Omega\to \s^1$, harmonic away from $|\deg(\varphi)|$ singularities and distributionally harmonic everywhere, thus emphasising the relevance of \eqref{eq: GL} as a variational relaxation to the (non-variational) problem of finding harmonic maps with prescribed point singularities. Furthermore, they showed that maps with these properties exist for any prescribed singular set and boundary datum, under a topological compatibility condition. They call such maps \emph{canonical harmonic maps}.

\medskip

 The position of these singularities is mediated by an interaction energy $W$ between the points of the singular set --- called renormalised energy --- obtained by the desingularised limit
\begin{equation*}
	W=\lim_{\epsilon\to 0}\left(E_\epsilon(u_\epsilon)-2\pi|\deg(\varphi)||\log\epsilon |\right).
\end{equation*}
The same energy appears also in the context of $p$-harmonic maps, with $p<2$, again as a desingularised limit of the energy of minimisers \cite{Hardt-Lin1995}. This suggests that the renormalised energy is really a part of the (infinite) energy of singular $\s^1$-valued harmonic maps, that manifests only when the Dirichlet energy is relaxed in some way.

\medskip

There are several directions in which this result can be generalised. For instance, an analogous interaction energy was found for point singularities in higher codimension via $p$-harmonic relaxation \cite{Hardt-Lin-Wang1997} or for general target manifolds \cite{Monteil-Rodiac-VanSchaftingen2021,Monteil-Rodiac-VanSchaftingen2022,VanSchaftingen-VanVaerenbergh2023}. Ignat and Jerrard \cite{Ignat-Jerrard2021} showed the validity of a similar result for unimodular vector fields on a closed Riemannian surface with a finite number of prescribed singular points, detecting a very interesting interplay with the genus of the surface. For higher dimensional singularities, this interaction is understood in the case of nearly parallel vortex filaments in the three dimensional Euclidean space $\R^3$ via Ginzburg--Landau \cite{DelPinoKowalczyk2008,ContrerasJerrard2017, Davila-delPino-Medina-Rodiac2022}.
 
\medskip

In this work, we advance this program with a twofold goal:

\begin{itemize}
	\item Give a complete classification of the set of $\s^1$-valued maps on a closed Riemannian manifold $(M^n,g)$ with prescribed (2-codimensional) singular set, classically harmonic away from the singularities and distributionally harmonic on the entire $M$. 
	\item Describe the renormalised energy of the singular set of $\s^1$-valued maps with higher-dimensional singularities, and its role in three different relaxations of the Dirichlet energy.
\end{itemize}
For instance, we establish that the interaction between two embedded, disjoint $C^{1,1}$ curves $\gamma_1,\gamma_2\subset \R^3$ is given by 
\begin{equation}\label{eq: interaction}
	\int_{\gamma_1}\int_{\gamma_2}\frac{\dot\gamma_1(x)\cdot \dot\gamma_2(y)}{|x-y|}\d\mathcal{H}^1(x)\d\mathcal{H}^1(y),
\end{equation}
which is consistent with the known formula for nearly parallel vortex filaments. This answers a question by Rivi\`ere \cite{Riviere1996}. Here $\dot\gamma_k$ is the unit tangent vector to the curve; it can be noted from the numerator that orthogonal parts of the filament do not interact, while the interaction is strongest in the parallel case, with a sign that depends on the orientation of the curve. The full energy includes also the interaction between different parts of a single curve, obtained by a desingularisation of the same formula.  We also remark that interaction \eqref{eq: interaction} corresponds exactly to the magnetic inductances generated by electric currents running with fixed intensities in a prescribed set of  wires $\gamma_1,\dots,\gamma_m$, where the direction of the current determines the orientation. This can be shown simply by Maxwell's equations, see Appendix \ref{app: ren en Rn}.

\subsection{Main results}
The first result of this work is a complete characterisation of the $\s^1$-valued harmonic maps on a manifold with prescribed singular set. The multiplicity of these maps (up to constant phases) is determined by the topology of the ambient space $M$, in particular by the non-vanishing of the first Betti number $b_1(M)$, counting the one dimensional ``holes'' of $M$.

\medskip

\noindent{\bf Theorem A.} Let $\Gamma$ be a smooth submanifold of $(M,g)$ which is the boundary of an orientable hypersurface $\Sigma$. Then there is a collection of maps $u_\alpha\colon M\to \s^1$, parametrised by $\alpha\in H^1(M,2\pi\Z)$, such that 
\begin{itemize}
	\item $\Sing(u_\alpha)=\Gamma$,
	\item $u_\alpha$ is harmonic on $M\setminus\Gamma$,
	\item $u_\alpha$ is distributionally harmonic on $M$.
\end{itemize}
Moreover, if any other map satisfies these three conditions, it is equal to one of the $u_\alpha$s up to a constant phase.

\smallskip
This result is a consequence of Theorem \ref{thm: harmonic maps} below, whose statement is analogous but with $\Gamma$ being in the more general class of integral currents. Denoting with $\mathscr{H}(M,\Gamma)$ the set of harmonic maps arising from Theorem A, we make a few observations:
\begin{itemize}
	\item The requirement that $\Gamma=\partial\Sigma$ is necessary for the existence of a regular map $u\colon M\setminus \Gamma\to\s^1$ with singular set $\Gamma$, see for instance \cite[Theorem 3.8]{Alberti-Baldo-Orlandi2003}.
	\item The singular maps $u_\alpha\in \mathscr{H}(M,\Gamma)$ belong naturally to the space $B2V$ of functions with bounded higher variation defined by Jerrard and Soner \cite{Jerrard-Soner2002}.
	\item Even in the planar case, topology breaks uniqueness. This can be already seen in multiply connected planar domains: in the annulus $B_2\setminus \overline{B_1}$, any singular harmonic maps $u$ can be replaced by $ue^{if_k(|x|)}$, where $f_k$ solves 
\begin{equation*}
	\begin{cases}
		f_k''+\tfrac1r f_k'=0& \text{in }(1,2)\\
		f_k(1)=0,\ f_k(2)=2k\pi,
	\end{cases}
\end{equation*}
without changing harmonicity, singular set or boundary datum of $u$.
In the same way, in Theorem A all possible choices of harmonic maps are parametrised by the ``one dimensional topology'' of the ambient, represented by the integral cohomology $H^1(M,2\pi\Z)$.
	\item It is possible to desingularise \emph{canonically} each $u_\alpha\in\mathscr{H}(M,\Gamma)$ and make sense of a (further) renormalised energy $e(u_\alpha)$. This notion could be in principle useful to define a ``canonical'' harmonic map as the least energetic one in $\mathscr{H}(M,\Gamma)$, but uniqueness of the minimisers generally fails (see Remark \ref{rem: canonical} below).
\end{itemize}

\medskip

In the second part of this work, we specialise to the case in which $\Gamma$ is a $C^{1,1}$ embedded submanifold of codimension 2, and we study three variational relaxations of the problem of finding harmonic maps with prescribed singular set. 
\begin{enumerate}
	\item Given $\delta>0$ let $T_\delta$ be a $\delta$-tubular neighbourhood of $\Gamma$ and $M_\delta\coloneqq M\setminus T_\delta$. We consider the minimisation problem 
	\begin{equation}\label{eq: intro delta}
		\inf\int_{M_\delta}|\d u|^2,
	\end{equation}
	among all $u\in W^{1,2}(M_\delta,\s^1)$ with a fixed nonzero degree around $\Gamma$. Prescribing the degree ``forces the singularity'' in the limit $\delta\to 0$.
	\item For $p<2$, we consider
	\begin{equation*}
		\inf\int_{M}|\d u|^p
	\end{equation*}
	among all maps $W^{1,p}(M,\s^1)$ with prescribed singular set, in the sense of Definition \ref{def: singular set} below.
	\item For $s\in(\tfrac12,1)$, we consider
	\begin{equation*}
		\inf\, [u]^2_{H^s(M)}
	\end{equation*}
	among all maps $H^s(M,\s^1)$ with prescribed singular set, in the sense of Definition \ref{def: singular set Hs} below.  
\end{enumerate}
The singular set of $u$ in Definitions \ref{def: singular set} and \ref{def: singular set Hs} is defined in terms of its distributional Jacobian $Ju$. This is not equivalent to minimising among functions that are smooth outside of the prescribed singularity, since $Ju$ reads only the topologically meaningful part of the singular set (see Remark \ref{rem: ABO} below).

\medskip

We prove that minimisers $u_\delta$, $u_p$ and $u_s$ for the above problems exist\footnote{Plain existence of $u_p$ and $u_s$ follows by direct method, since the map $u\mapsto Ju$ is weakly sequentially continuous in both topologies, see \cite[Proposition 8.1]{Brezis-Mironescu2021}.}, and we give a precise description, using $\mathscr{H}(M,\Gamma)$. Each of the above problem is in fact equivalent to two nested minimisations. First, the energy is minimised among all maps that can be expressed as $e^{i\varphi}u_\alpha$ for a fixed $u_\alpha\in \mathscr{H}(M,\Gamma)$ and $\varphi\in W^{1,2}(M_\delta,\R)$ (resp.~ $W^{1,p}(M,\R)$, $H^s(M,\R)$), which is a scalar problem. Then, we minimise the result among all $u_\alpha\in \mathscr{H}(M,\Gamma)$, which is a discrete problem. In other words, space of competitors can be broken down in ``topological sectors'', where every element differs from a given harmonic map by a phase. This allows to give a very precise description of the minimisers in terms of the ``optimal'' singular harmonic map, namely the element of  $\mathscr{H}(M,\Gamma)$ whose desingularised energy is the lowest. This last fact is a consequence of the following asymptotic expansions, which is the second result of this work.

\medskip

\noindent{\bf Theorem B.} There is an interaction energy $W_M(\Gamma)$ between different parts of the singular set such that 
 \begin{align*}
	\int_{M_\delta}|\d u_\delta|^2&=2\pi\mathcal{H}^{n-2}(\Gamma)\log(1/\delta)+W_M(\Gamma)+e(u_\alpha)+o(1)\\
	\int_{M}|\d u_p|^p&=\frac{2\pi\mathcal{H}^{n-2}(\Gamma)}{2-p}+W_M(\Gamma)+e(u_\alpha)+o(1)\\
	[u_s]^2_{H^s(M)}&=\frac{2\pi\mathcal{H}^{n-2}(\Gamma)}{2-2s}+W_M(\Gamma)+e(u_\alpha)+o(1)
\end{align*}
as $\delta\to 0$, $p\to 2$ and $s\to 1$, respectively. Here $e(u_\alpha)$ is the desingularised energy of the harmonic map $u_\alpha\in \mathscr{H}(M,\Gamma)$ belonging to the same topological sector as $u_\delta$ (resp.~ $u_p$, $u_s$).
 
\medskip

Theorem B is a consequence of Theorems \ref{thm: expansion udelta} and \ref{thm: expansion p s} below. We make three further remarks.
\begin{itemize}
	\item The interaction energy $W_M(\Gamma)$ depends exclusively on the singular set and the ambient metric; not on the topological sector of the minimisers.
	\item Given the above expansions for the $W^{1,p}$ and $W^{s,2}$ relaxations, it is natural to conjecture that a similar formula holds for mixed-type $W^{s,p}$ relaxations, where $sp\nearrow 2$. In particular, if $u_{s,p}$ minimises the $W^{s,p}$ seminorm among all $W^{s,p}(M,\s^1)$ maps with prescribed distributional Jacobian, we expect that 
\begin{equation*}
	[u_{s,p}]_{W^{s,p}(M)}^p=\frac{2\pi\mathcal{H}^{n-2}(\Gamma)}{2-sp}+W_M(\Gamma)+e(u_\alpha)+o(1)
\end{equation*}
as $sp\to 2$.
\item A direct consequence of Theorem B is an asymptotic expansion for the least energy to pay to prescribe the distributional Jacobian \cite[\S1.8 and \S8.3]{Brezis-Mironescu2021}. More precisely, the energies of $u_p$ and $u_s$ correspond to 
\begin{equation*}
	{S}_{p,s}(v)\coloneqq\inf\{\, [u]_{W^{s,p}}^p:Ju=Jv\}
\end{equation*} 
in the cases $(s,p)=(1,p)$ and $(s,2)$, respectively. Here, $v$ is any map whose distributional Jacobian describes the singular set $\Gamma$ --- for instance any $v\in \mathscr{H}(M,\Gamma)$.
\end{itemize}

\subsection{Area in higher codimension via singular limits: an outlook}

The connection between relaxations of harmonic maps and the area of the singular set has been long known. In the Ginzburg--Landau context, this connection has been explored, among many works, in \cite{Bethuel-Brezis-Helein1993, Bethuel-Brezis-Helein1994, Struwe1994, Jerrard-Soner2002CalcVar, Jerrard-Soner2002JFA, Alberti-Baldo-Orlandi2005,Mesaric2009,Stern2021, Colinet-Jerrard-Sternberg2021, DePhilippis-Pigati2024} via $\Gamma$-convergence and energy concentration. Similar results have been achieved for $p$-harmonic maps, see the works \cite{Hardt-Lin1987,Hardt-Lin1995,Hardt-Lin-Wang1997,Stern2020}. 

\medskip

In the codimension one case, the Allen--Cahn model provides natural analogue of \eqref{eq: GL}. In the works of Guaraco and Gaspar--Guaraco \cite{Guaraco2018,Gaspar-Guaraco2018} this model was successfully employed for the construction of critical points of the area functional via min-max on the equation, rather than on the surface itself, building on the prior regularity theory of \cite{Hutchinson-Tonegawa2000,Tonegawa2005,Tonegawa-Wickramasekera2012}. This program was brought forward in three dimensions in \cite{Chodosh-Mantoulidis2020}, where the authors managed to perform a delicate passage to the limit $\eps\to 0$ and showed the validity of the multiplicity one conjecture, as well as reproving Yau's conjecture on the existence of infinitely many minimal surfaces in closed 3-manifolds.

\medskip

A very promising alternative to the Allen--Cahn equation is given by nonlocal minimal surfaces. Initially introduced in \cite{Caffarelli-Roquejoffre-Savin2010}, those objects recently received a renewed interest in a program of classical minimal hypersurface approximation via nonlocal minmax \cite{Chan-Dipierro-Serra-Valdinoci2023, Caselli-Florit-Serra2024Yau,Florit2024}. 

\medskip

It seems only natural then to try to apply the same min-max scheme also in higher codimensional cases, for instance by means of Ginzburg--Landau or $p$-harmonic relaxation. The problem arising from this approach is twofold: first, the natural uniform bound on the energy is not enough to guarantee concentration into an integral varifold, as shown in \cite{Pigati-Stern2023}. Even worse, if the ambient manifold has nontrivial one-dimensional topology (e.g.~ $M=N\times\s^1$) then the min-max procedure could pick up solutions without any nodal/singular set, generating a diffuse limit, see \cite{Stern2020,Stern2021}. An alternative approach is given by the Yang--Mills--Higgs equation, a modification of Ginzburg--Landau with features much closer to Allen--Cahn, which has proven to be a successful tool in the study of area in codimension two \cite{Riviere1996, Pigati-Stern2021, Parise-Pigati-Stern2024, Badran-DelPino2023, Badran-DelPino2024, MarxKuo2023, Liu-Wei-Ye2024, Liu-Ma-Wei-Wu2021}.

\medskip

Recently, Serra \cite{Serra2024} proposed another method, generalising the approximation by $s$-minimal hypersurfaces, based on a notion of fractional area in codimension $k\geq 2$. Given a smooth, codimension $k$ submanifold $\Gamma$ and $s\in(0,1)$ we define its fractional mass as 
\begin{equation}\label{eq: fractional mass}
	\mathbb{M}^{n-k}_s(\Gamma)\coloneqq \inf\, [u]_{H^{(k-1+s)/2}(M)}^2
\end{equation}
where the infimum is taken over all maps $u\colon M\setminus\Gamma\to\s^{k-1}$ that are \emph{linked} with $\Gamma$ and $H^{(k-1+s)/2}(M)$ is the fractional Sobolev space on a closed manifold $M$, introduced in  \cite{Caselli-Florit-Serra2024}. Fractional minimal surfaces of codimension $k$ as the ones who are critical for \eqref{eq: fractional mass}. This definition can be seen as a natural generalisation of nonlocal minimal surfaces, when thought as singular set of critical points of the $H^{s/2}$ seminorm among functions with values in $\s^0=\{-1,+1\}$. The minimiser in \eqref{eq: fractional mass} is a fractional harmonic map, first introduced in \cite{DaLio-Riviere2011}. The linking condition enforces a singularity along $\Gamma$ that weights asymptotically as the mass $\mathcal{H}^{n-k}(\Gamma)$. This was made precise in \cite{CaselliFregugliaPicenni2024}, where the authors proved $\Gamma$-convergence in the case $k=2$. 

\medskip

This approach is potentially more powerful than a direct min-max on the energy. First starting from the submanifolds rules out problems related to non-integrality of the limit. Secondly, the maps arising have the property of being minimising outside of their singular set, which ensures that the minimisers in \eqref{eq: fractional mass} cannot ``warp'' too many times around a one-dimensional hole, which in turn seems to rule out potential diffuse limits. 

\medskip

If the minimisation \eqref{eq: fractional mass} is taken among $H^s$ competitors with the same singular set in the sense of the distributional Jacobian, then by definition the map $u_s$ arising from fractional relaxation with prescribed singular set $\Gamma$ is precisely the minimiser of \eqref{eq: fractional mass} in the case $k=2$, and its $H^s$ energy is the fractional mass. In particular, renormalising the third formula of Theorem B (multiplying by $(1-s)$) provides an asymptotic expansion of the fractional area of $\Gamma$ and reveals a lower order interaction between different components of $\Gamma$ itself, which vanishes in the local case $s=1$. This fact is also consistent with the codimension 1 case, where the interaction between different hypersurface sheets is described by the Davila--del Pino--Wei system \cite{Davila-delPino-Wei2018}. Moreover, the above approach could be extended to $W^{1,p}$ or, more generally, $W^{s,p}$ notions of mass, defining
\begin{equation*}\label{eq: fractional mass}
	\mathbb{M}^{n-k}_{s,p}(\Gamma)\coloneqq \inf\, [u]_{W^{s,p}}^p
\end{equation*} 
where $s>0$, $1\leq p<\infty$, $sp<k$ and the infimum is taken over all $u\in W^{s,p}(M,\s^1)$ with singular set $\Gamma$. The second formula in Theorem B then gives an asymptotic of $\mathbb{M}_{1,p}^{n-2}(\Gamma)$ as $p\to 2$.

\subsection{Plan of the paper} The paper is organised as follows:
\begin{itemize}
	\item In Section \ref{sec: currents} we recall several notions that will be helpful throughout the paper, such as differential forms, currents, the Hodge decomposition, the Hodge laplacian and the distributional Jacobian. 
	\item In Section \ref{sec: HM} we introduce the set $\mathscr{H}(M,\llb\Gamma\rrb)$ of harmonic maps to the circle which are distributionally harmonic on $M$, singular in $\Gamma$ and classically harmonic outside of their singular set, modulo constant phases. We relate the multiplicity of those maps to the topology of $M$, giving an explicit parametrisation via the integral cohomology group $H^1(M,2\pi\Z)$. In the embedded, $C^{1,1}$ case, we also define the renormalised energy $W_M(\Gamma)$ and show that the energy of the harmonic maps on the complement of a $\delta$-neighbourhood of $\Gamma$ features $W_M(\Gamma)$ at second order.
	\item In Section \ref{sec: minimisers} we introduce the three variational relaxations of the harmonic map with prescribed singularity problem. Using the class $\mathscr{H}(M,\llb\Gamma\rrb)$, we show that the whole variational problem can be broken down in the two nested minimisations described above. In order to show convergence to the energy expansion, we proceed in two steps:
	
	\smallskip
	
	\begin{itemize}
	\item[1.] We prove directly the expansion for $u_\delta$, showing that its energy is close to the energy of a map in $\mathscr{H}(M,\llb\Gamma\rrb)$, for which we already know the expansion. This requires a regularity theory near $\partial M_\delta$, which is obtained through a De Giorgi iteration scheme.
	\item[2.] We prove the expansion for $u_p$ and $u_s$ via comparison with $u_\delta$, using the fact that the problems they minimise are, in a way, ``close''. 
	\end{itemize}
	
	\item In Appendix \ref{app: ren en Rn}, we give explicit expressions of the renormalised energy in $\R^n$ and explain how it related with the theory of electromagnetism.
	\end{itemize}
\bigskip

\noindent{\bf Acknowledgments:} The author is truly grateful to Joaquim Serra for his constant support and valuable advice. He also thanks Jaume de Dios and Gerard Orriols for many interesting discussions and helpful comments. This project has received funding by the European Research Council under Grant Agreement No 948029.

\section{Currents and the distributional Jacobian}\label{sec: currents}
\subsection{Integral currents and the Hodge Laplacian}
Let $(M,g)$ be a smooth, closed, oriented $n$-dimensional Riemannian manifold. For any $0\leq k\leq n$, we denote with $\Lambda^k(M)$ the vector bundle on $M$ of alternating $k$-linear functionals on $TM\times \cdots \times TM$ ($k$ factors) and with $D^k(M)$ the space of smooth, compactly supported sections $M\to \Lambda^k(M)$, namely the space of differential $k$-forms on $M$. Let $D_k(M)$ be the space of $k$-currents on $M$, namely the space  continuous linear functionals on $D^k(M)$, endowed with the usual locally convex topology (see e.g.~ \cite{Simon1983}). We denote with $\d\colon D^k(M)\to D^{k+1}(M)$ and $\d^*\colon D^k(M)\to D^{k-1}(M)$ the exterior differential and codifferential operators, respectively. If $\alpha,\beta\in D^k(M)$ we denote their pointwise product via the metric $g$ with $(\alpha,\beta)_g$. Lastly, the Hodge star operator $\star\colon D^k(M)\to D^{n-k}(M)$ is defined via the pairing
\begin{equation*}
	\int_M\alpha\wedge\star\beta=\int_M(\alpha,\beta)_g\operatorname{vol}_g.
\end{equation*}

\medskip

To any $k$-dimensional oriented, smooth submanifold $\Gamma$ of $M$ we can associate a $k$-current $\llbracket \Gamma\rrbracket\in D_k(M)$ by
\begin{equation*}
	\llbracket \Gamma\rrbracket(\alpha)\coloneqq \int_\Gamma\alpha=\int_\Gamma(\alpha(x),\vec \Gamma(x))_g \d\mathcal H_g^k(x).
\end{equation*}
Here, $\mathcal{H}^k_g$ is the $k$-dimensional Hausdorff measure on the manifold $(M,g)$ and $\vec \Gamma$ is the orientation of $\Gamma$, namely a continuous section $\Gamma\to \Lambda^k(\Gamma)$ such that $\vec \Gamma= \alpha_1\wedge\cdots\wedge \alpha_k$ on every $x\in\Gamma$, being $\alpha_1,\dots,\alpha_k$ an orthonormal basis of $T^*_x\Gamma$. 

\medskip

We say that a $k$-current $T$ is integer rectifiable if there exists a $\mathcal{H}_g^k$-measurable rectifiable set $\Gamma$, a locally $\mathcal{H}_g^k$-measurable positive integer-valued function $\theta$ and an $\mathcal{H}_g^k$-measurable map $\xi\colon \Gamma\to \Lambda^{k}(M)$ that can be represented as $\alpha_1\wedge\dots\wedge \alpha_k$, where $\{\alpha_1,\dots,\alpha_k\}$ form a basis for the approximate tangent space a.e.~ in $\Gamma$,   such that 
\begin{equation*}
	T(\alpha)=\int_\Gamma\theta(x)(\alpha(x),\xi(x))_g \d\mathcal H_g^k(x)
\end{equation*}
The notion of boundary of a current is obtained by enforcing Stokes' theorem. If $T\in D_{k}(M)$, we define $\partial T\in D_{k-1}(M)$ via the pairing 
\begin{equation*}
	\partial T(\alpha)\coloneqq T(\d \alpha)\quad\forall \alpha\in D^{k-1}(M).
\end{equation*}
A current $T$ is said to be integral if both $T$ and $\partial T$ are integer rectifiable.

\medskip

We will denote with $\Delta=\d^*\d+\d\d^*\colon D^k(M)\to D^k(M)$ the Hodge Laplacian\footnote{Note that this is the \emph{positive spectrum} Laplacian; on $0$-forms (i.e.~ functions) in $\R^n$ it is given by $\Delta=-\sum_{j=1}^n\frac{\partial^2}{\partial x_j^2}$.} on forms. The same operator can be defined on $D_k(M)$ by distributional pairing, i.e.~ $\Delta T(\alpha)\coloneqq T(\Delta\alpha)$. An element $\omega\in D_k(M)$ is said to be harmonic if $\Delta\omega=0$ or, equivalently, if 
\begin{equation*}
	\d\omega=0\quad\text{and}\quad \d^*\omega=0.
\end{equation*}
By Weyl's lemma, an harmonic $k$-current is also an harmonic $k$-form, and the vector subspace of harmonic forms is denoted with $\mathcal{H}^k_\Delta(M)\subset D^k(M)$. Finally, we recall that any $\alpha\in D^k(M)$ enjoys a Hodge decomposition: there exists $\varphi\in D^{k-1}(M)$, $\psi\in D^{k+1}(M)$ and $\omega\in \mathcal{H}^k_\Delta(M)$ such that
\begin{equation*}
	\alpha=\d\varphi+\d^*\psi+\omega
\end{equation*}
and the same is true for $W^{1,p}$ forms, see \cite{Schwarz1995}.

\medskip

We can define a \emph{fundamental solution} for the Hodge Laplacian.
For a $k$-form $\eta\in D^k(M)$ we want to find a solution to the equation
\begin{equation*}
	\Delta^k\omega=\eta\quad\text{on }M,
\end{equation*}
where we are denoting with $\Delta^k$ the Hodge Laplacian to highlight that it acts on $k$-forms.
Since $\Delta^k$ is a Fredholm operator with kernel $\mathcal{H}^k_\Delta$, standard Fredholm theory guarantees the existence of a solution to 
\begin{equation*}
	\Delta^k\omega=\eta^\perp \coloneqq \eta- \eta^\top
\end{equation*}
where $\eta^\top$ is the $L^2$ projection of $\eta$ onto $\mathcal{H}^k_\Delta$.
The solution is unique in $\mathcal{H}^k_\Delta(M)^\perp\coloneqq \{\omega\in D^k(M):\omega^\top=0\}$. Such solution has an integral representation via a ``fundamental form'': let $\Lambda^k(M)\boxtimes \Lambda^k(M)$ be the subspace of $\Lambda^{2k}(M\times M)$ of elements of the form
\begin{equation}\label{eq: boxtimes form}
	\sum_{I,J}f_{I,J}(x,y)\d x^I\wedge\d y^J,
\end{equation}
where the sum is taken over $I,J\in\mathcal{I}_k\coloneqq \{(i_1,\dots,i_k):1\leq i_0<\cdots<i_k\leq n\}$ and $\d x^I\coloneqq \d x^{i_1}\wedge\cdots\wedge \d x^{i_k}$. The fundamental solution or ``Green's operator'' for $\Delta^k$ is a (distributional) section $G^k_M\in D'(M\times M,\Lambda^k(M)\boxtimes \Lambda^k(M))$ characterised by the action 
\begin{equation}\label{eq: action Green form}
	\int_M\langle G_M^k(x,y),\Delta^k_y\alpha(y)\rangle_y dy=\alpha(x)
\end{equation}
for any $k$-form $\alpha\in \mathcal{H}^k_\Delta(M)^\perp$, see for instance \cite{Warner1983,Scott1995}. By representing $G_M^k$ locally in a small coordinate patch it's easy to see that it can be represented by \eqref{eq: boxtimes form} with $f_{I,J}\in L^1_{\text{loc}}$, since in $\R^n$ the Hodge Laplacian is nothing but the standard Laplacian on each component. Then, the action $\langle\cdot,\cdot\rangle_y$ means
\begin{equation*}
	\left\langle f_{I,J}(x,y)\d x^I\wedge\d y^J,\beta_K(y)\d y^K\right\rangle_y\coloneqq f_{I,J}(x,y)\beta_K(y)(\d y^J,\d y^K)_g\d x^I
\end{equation*} 
For a general $k$-form $\alpha$, the right-hand side of \eqref{eq: action Green form} is projected onto $\mathcal{H}^k_\Delta(M)^\perp$, namely
\begin{equation*}
	\int_M\langle G_M^k(x,y),\Delta^k_y\alpha(y)\rangle_y dy=\alpha^\perp(x).
\end{equation*}
In particular, for any $\beta\in \mathcal{H}^k_\Delta(M)^\perp$ there exists a unique $\alpha\in \mathcal{H}^k_\Delta(M)^\perp$ such that $\Delta^k\alpha=\beta$ and such $\alpha$ can be represented as
\begin{equation*}
	\alpha(x)=\int_M\langle G_M^k(x,y),\beta(y)\rangle_y \d y.
\end{equation*}
\subsection{The distributional Jacobian}
The notion we use to describe the singular set of a map $u\colon M\to \s^1$ is its distributional Jacobian  $Ju$. For $W^{1,1}(M,\s^1)$ maps, the real-valued 1-form $ju\coloneqq u^*(\d \theta)=(\d u,iu)=u^1\d u^2-u^2\d u^1$ belongs to $L^1$, where we denoted with $(\cdot,\cdot)$ the inner product in $\C$. Note that $ju$ is the projection of $\d u$ on its tangential component to $\s^1$ --- the only nonvanishing component since $(\d u,u)=\tfrac12\d|u|^2=0$ --- and thus
\begin{equation*}
	|\d u|=|ju|
\end{equation*}
pointwise a.e. We record the simple algebraic properties
\begin{equation}\label{eq: algebraic properties j}
	j(e^{i\varphi})=\d\varphi,\quad j(uv)=ju+jv,\quad j(u^{-1})=-ju,
\end{equation} 
valid in $W^{1,1}(M,\s^1)$. The distributional Jacobian $Ju$ is defined as the distributional exterior differential of $ju$, that is
\begin{equation}\label{eq: dist jac}
	\langle Ju,\zeta\rangle\coloneqq \int_M (ju,\d^*\zeta)_g
\end{equation}
for any smooth $2$-form $\zeta$. For regular maps, $Ju=2\d u^1\wedge \d u^2$ and by \eqref{eq: algebraic properties j} it follows that 
\begin{equation}\label{eq: algebraic properties J}
	 J(e^{i\varphi})=0,\quad J(uv)=Ju+Jv,\quad J(u^{-1})=-Ju.
\end{equation}
It is well known (see e.g.~ \cite{Brezis-Mironescu2021, Alberti-Baldo-Orlandi2003}) that the distributional Jacobian tracks the singularities of $u$ \emph{as currents}, meaning that $Ju$ carries not only set-theoretic informations about $\Sing(u)$, but also orientation and multiplicity. As an example, if $u(x)=(x/|x|)^k$ in $B_1\subset \R^2$, then 
\begin{equation*}
	Ju=(2\pi k\delta_0)\d x^1\wedge \d x^2.
\end{equation*}
\begin{definition}\label{def: singular set}
	We say that the singular set of an $\s^1$-valued map $u\in W^{1,1}$ is the current $T$ if $Ju=2\pi \star T$, where $\star$ is the Hodge star operator.
\end{definition}
Definition \ref{def: singular set} is motivated by the following fact. If $\Gamma$ is embedded and $u\in W^{1,1}(M,\s^1)\cap C(M\setminus \Gamma)$, we can define the degree $\deg(u,\Gamma)$ of $u$ around $\Gamma$ as the degree of the restriction of $u$ to any simple loop around $\Gamma$, oriented positively with the orientation of $\Gamma$\footnote{To be more precise, the normal bundle $N\Gamma$ has a natural orientation determined by the orientation of $\Gamma$. for any $x\in\Gamma$ any sufficiently small disc in $N_x\Gamma$ can be mapped in $M$ via the exponential map, without intersecting $\Gamma$ if not in $x$ itself. Then, $\deg(u,\Gamma)$ is just the degree of $u$ in the image of any small circle $\partial B_\eps\subset N_x\Gamma$ via the exponential map, and by continuity the choice is independent both on $x\in \Gamma$ and $\eps>0$.}. Then we have 
\begin{equation*}
	Ju=2\pi \deg(u,\Gamma)\star \llb\Gamma\rrb,
\end{equation*}
see for example \cite[Theorem 3.1]{Brezis-Mironescu2021}.

Using Definition \ref{def: singular set} we can define the singular set of a map in $W^{1,p}$, but we are also interested in doing that for $H^s$ maps for which the quantity \eqref{eq: dist jac} is not necessarily finite. However, for $s>1/2$ the distributional Jacobian can still be defined by means of the Factorisation theorem \cite[Theorem 7.1]{Brezis-Mironescu2021}: every map $u\in H^s$ can be written as $u=e^{i\phi}v$ for some $v\in W^{1,2s}$ and $\phi\in H^s$. Then, we define
\begin{equation*}
	\widetilde Ju\coloneqq Jv.
\end{equation*}
\begin{definition}\label{def: singular set Hs}
	We say that the singular set of an $\s^1$-valued map $u\in H^s$, $s\in (\tfrac12,1)$, is the current $T$ if $\widetilde Ju=2\pi \star T$.
\end{definition}
In the rest of this work, we will consider only currents that are admissible according to the following definition.
\begin{definition}\label{def: admissibility}
	We say that the current $T$ is admissible if $T=\partial S$, where $S$ is an integral $(n-1)$-current. In particular, $T$ is integer rectifiable. 
\end{definition}

\medskip

For the sake of readability, we will specialise to the case where the multiplicity of $T$ is one, since very little modifications are needed to include higher multiplicity. In this case, we suppose that $T$ is supported on a rectifiable set $\Gamma$ with orientation $\vec\Gamma$ and we denote the current $T\coloneqq\llbracket\Gamma\rrbracket$.

\begin{remark}\label{rem: ABO}
	The interpretation of the Jacobian as the singular set of a Sobolev $\s^1$-valued map was extensively studied by Alberti, Baldo and Orlandi in the beautiful paper \cite{Alberti-Baldo-Orlandi2003}. As they remark, this interpretation must be handled with care; for instance, a vanishing distributional Jacobian does \emph{not} imply regularity of the function --- a point singularity in $\R^3$ will be completely ignored by $J$. What is true is that the Jacobian tracks the part of the singular set which is \emph{topologically relevant}.
\end{remark}

We conclude this section with a simple remark related to the harmonicity of $u$.
\begin{remark}
	The codifferential of $ju$ measures the harmonicity of $u$, in the following sense: for any smooth map $u\colon M\to \s^1$
\begin{equation}\label{eq: codiff ju}
	\d^*ju=\d^*(\d u,iu)=(\d^*\d u,iu)+(\d u,i\d u)=-(\Delta_Mu,iu);
\end{equation}
that is, any map $u\colon M\to \s^1$ is harmonic if and only if $\d^*ju=0$. Moreover, $u\in W^{1,2}(M,\s^1)$ is weakly harmonic if and only if $\d^*ju=0$ weakly. Indeed, consider an outer variation of the form 
\begin{equation*}
	u_t=\frac{u+tv}{|u+tv|}
\end{equation*}
for some smooth perturbation $v$. A direct expansion shows that $u_t=u+t\hat v\,iu+\O(t^2)$, where $\hat v=(v,iu)$. Then 
\begin{align*}
	\frac{d}{dt}\Big\vert_{t=0}\int_M|\d u_t|^2&=\frac{d}{dt}\Big\vert_{t=0}\int_M|j u_t|^2\\
	&=\int_M ju\cdot \d\hat v.
\end{align*}
\end{remark}

\section{The class $\mathscr{H}(M,\llb\Gamma\rrb)$ and the energy expansion}\label{sec: HM}

Given a singular map $u\in W^{1,p}(M,\s^1)$, for some $p\in(1,2)$, we can decompose the real-valued one form $ju$ as
\begin{equation}\label{eq: Hodge decomposition}
	ju=\d\phi+\d^*\psi+\omega
\end{equation}
where $\omega$ is harmonic and the expression holds in the sense of distributions. Expression \eqref{eq: Hodge decomposition} is particularly meaningful because it highlights how the harmonicity of $u$ and its singular set are, in a way, independent. Indeed, taking the codifferential of \eqref{eq: Hodge decomposition} and using \eqref{eq: codiff ju}  we find 
\begin{equation*}
	(-\Delta_Mu,iu)=\d^*ju=\d^*\d \phi=-\Delta_M\phi.
\end{equation*}
In other words, $u$ is harmonic to $\s^1$ if and only if $\phi$ is an harmonic function. On the other hand, prescribing a singular set $\llbracket\Gamma\rrbracket$ in the sense of Definition \ref{def: singular set} reads
\begin{equation}\label{eq: prescr. sing. set}
	2\pi\star\llbracket\Gamma\rrbracket=Ju=\d ju=\d\d^*\psi.
\end{equation}
We now show that \eqref{eq: prescr. sing. set} is solvable for every admissible singular set $\llbracket\Gamma\rrbracket$.

\begin{lemma}\label{lem: solvability psi}
	Let $\llbracket\Gamma\rrbracket$ be admissible in the sense of Definition \ref{def: admissibility}. Then, there exists a unique $\psi=\psi(\llbracket\Gamma\rrbracket)\in D_{2}(M)\cap \mathcal{H}^{2}_\Delta(M)^\perp$ such that \eqref{eq: prescr. sing. set} holds. 
	\begin{proof}
		As a consequence of admissibility we find that $\llbracket\Gamma\rrbracket=\llbracket\partial \Sigma\rrbracket=\d^*\llbracket\Sigma\rrbracket$, namely that $\llb\Gamma\rrb$ is distributionally a coboundary and reveals that the projection of $\llb\Gamma\rrb$ onto the space of harmonic $(n-2)$-forms vanishes. Thus, defining 
\begin{align*}
	A(x)\coloneqq 2\pi\int_{M}\langle G_M^{n-2}(x,\cdot),\llbracket\Gamma\rrbracket\rangle= 2\pi\int_{\Gamma}\langle G_M^{n-2}(x,y),\vec\Gamma(y)\rangle_y \d\mathcal{H}^{n-2}_g(y)
\end{align*}
we find $\Delta A=2\pi\llbracket\Gamma\rrbracket$. Note that $A\in D_{n-2}(M)\cap \mathcal{H}^{n-2}_\Delta(M)^\perp$. We observe that, distributionally
\begin{equation}\label{eq: d*A harmonic}
	\Delta \d^*A=\d^*\d\d^*A=\d^*\Delta A=2\pi \d^*\llbracket \Gamma\rrbracket=0
\end{equation}
since $\d^*\llbracket \Gamma\rrbracket=\llbracket \partial \Gamma\rrbracket$=0; thus $\d^*A\in D_{n-3}(M)$ is distributionally harmonic, which implies in particular that
\begin{equation}\label{eq: dd* on w}
	\d\d^*A=0.
\end{equation}
Set $\psi=\star A\in D_2(M)\cap \mathcal{H}^{2}_\Delta(M)^\perp$. By \eqref{eq: dd* on w}, we have $\d^*\d\psi=0$
and this implies that 
\begin{equation*}\label{eq: d*dpsi}
	\d\d^*\psi=\Delta\psi=\star \Delta A=2\pi\star \llbracket\Gamma\rrbracket
\end{equation*}
which concludes the proof.
	\end{proof}
\end{lemma}
Next, we define and characterise the set of maps we are interested in.
\begin{definition}
	We denote with $\mathscr{H}(M,\llb\Gamma\rrb)$ the set of $\s^1$-valued maps on $M$ which are harmonic on $M\setminus\Gamma$, distributionally harmonic on $M$ and with singular set $\llb\Gamma\rrb$ in the sense of Definition \ref{def: singular set}, modulo constant phases.
\end{definition}
\begin{lemma}
	A map $u\in W^{1,p}(M,\s^1)$, with $p\in(1,2)$, belongs to $\mathscr{H}(M,\llb\Gamma\rrb)$ if and only if 
	\begin{equation}\label{eq: can harm map}
		ju=\d^*\psi+\omega
	\end{equation}
	for some harmonic one form $\omega$ and $\psi=\psi(\llb\Gamma\rrb)$ as in Lemma \ref{lem: solvability psi}
	\begin{proof}
		If $u\in \mathscr{H}(M,\llb\Gamma\rrb)$, then we can decompose $ju$ as\
		\begin{equation*}
			ju=\d\phi+\d^*\psi+\omega
		\end{equation*} 
		and observe that the decomposition is well defined in the equivalence class modulo constant phases, since changing representative amounts to replacing $\phi$ by $\phi+c$ for some constant $c\in\R$.
		Equation \eqref{eq: prescr. sing. set} and Lemma \ref{lem: solvability psi} imply that $\d^*\psi=\d^*\psi(\llb\Gamma\rrb)$, while global distributional harmonicity and \eqref{eq: codiff ju} implies that $\phi$ is distributionally harmonic, thus constant, so \eqref{eq: can harm map} holds. On the other hand, if \eqref{eq: can harm map} holds then the condition on the singular set and harmonicity away from the singular set are satisfied. Moreover, global distributional harmonicity holds if $\int_M(ju,\d\zeta)_g=0$ for every smooth function $\zeta\in D^0(M)$, which is clearly true.
	\end{proof}
\end{lemma}
\begin{remark}
	Without imposing the condition of distributional harmonicity on $M$, in principle the function $\phi$ in the Hodge decomposition could be any harmonic function in $M\setminus\Gamma$, potentially very singular on $\Gamma$, see for instance \cite{Bethuel-Brezis-Helein1994}.
\end{remark}

Next, we state and prove our first main result.
\begin{theorem}\label{thm: harmonic maps}
	Let $\llbracket\Gamma\rrbracket$ be admissible in the sense of Definition \ref{def: admissibility} and let $\psi=\psi(\llb\Gamma\rrb)$ be the map arising from Lemma \ref{lem: solvability psi}. Let $b_1(M)$ be the first Betti number, that is the rank of the first integral cohomology group $H^1(M,2\pi\Z)$. Then 
	\begin{enumerate}
		\item if $b_1(M)=0$ then $\mathscr{H}(M,\llb\Gamma\rrb)=\{u_\circ\}$ is a singleton.
		\item if $b_1(M)\ne 0$ there are multiple harmonic maps $\mathscr{H}(M,\llb\Gamma\rrb)=\{u_\alpha:\alpha\in H^1(M,2\pi\Z)\}$, satisfying 
		\begin{equation}\label{eq: equalpha th1}
			ju_\alpha=\d^*\psi+\omega_\alpha
		\end{equation} 
		where $\omega_\alpha= [\d^*\psi]+\alpha$, with $\alpha\in H^1(M,2\pi\Z)$. 
	\end{enumerate}
	\begin{proof}
	We start by observing that, for any map $u\colon M\to\s^1$, $ju=u^*(\d\theta)\in H^1(M,2\pi\Z)$ since $\d\theta\in H^1(\s^1,2\pi\Z)$, so the right-hand side of \eqref{eq: can harm map} must be integral (here and in what follows, when we say ``integral'' we always mean up to a factor of $2\pi$). 
	
	\medskip
	
		If $b_1(M)=0$, then we claim that $[\d^*\psi]$ belongs to the integral cohomology of $M\setminus\Gamma$, as a consequence of integrality of $\llb\Gamma\rrb$. Recall that the relative cohomology groups $H^k(M,M\setminus\Gamma;G)$, where $G$ is a given group, fit into a long exact sequence 
		\begin{equation}\label{eq: les}
		\begin{tikzcd}
		\cdots \arrow[r] & H^k(M,M\setminus\Gamma;G) \arrow[r, "j^*"]  & H^k(M;G) \arrow[r, "i^*"] & H^k(M\setminus\Gamma;G) \ar[out=-20, in=160, "\d"]{dll} \\ & H^{k+1}(M,M\setminus\Gamma;G) \arrow[r, "j^*"]  & H^{k+1}(M;G) \arrow[r, "i^*"] & H^{k+1}(M\setminus\Gamma;G) \rar[r] & \cdots 
		\end{tikzcd}
		\end{equation}
		where $i$ and $j$ are respectively the inclusion and the quotient map,
		see \cite[\S3.1]{Hatcher2002}. We consider the two sequences arising by choosing the integral $(G=2\pi\Z)$ and de Rham $(G=\R)$ cohomologies
		\begin{equation*}
		\begin{tikzcd}
		0 \arrow[r] & H^{1}(M\setminus\Gamma;2\pi\Z) \arrow[r, "\d"] \arrow[d] & H^{2}(M,M\setminus\Gamma;2\pi\Z) \arrow[r, "j^*"] \arrow[d] & H^{2}(M;2\pi\Z) \arrow[d] \\
		0 \arrow[r] & H^{1}(M\setminus\Gamma;\R) \arrow[r, "\d"] & H^{2}(M,M\setminus\Gamma;\R) \arrow[r, "j^*"] & H^{2}(M;\R)
\end{tikzcd}
		\end{equation*}
		where we used that $H^{1}(M,\R)=0$. Here the vertical arrows are the canonical maps $H^k(A,2\pi\Z)\to H^k(A,\R)$.
		The claim follows by diagram chasing; note first that $2\pi \star\llb\Gamma\rrb\in H^{2}(M,M\setminus\Gamma,2\pi\Z)$ is mapped to $0$ in $H^{2}(M;2\pi\Z)$, since $\Gamma$ is a boundary. By exactness of the sequence above, there is $\omega\in H^{1}(M\setminus\Gamma,2\pi\Z)$ such that $\d\omega=2\pi \star\llb\Gamma\rrb$. The same holds when mapped into de Rham cohomologies, but there we know that $\d\d^*\psi=2\pi \star\llb\Gamma\rrb$ as well (observe that, to make sense of this the sequence below needs to be taken in the distributional sense). By injectivity of $\d$, which follows by $H^{1}(M,\R)=0$, we get $\omega=\d^*\psi$, i.e.~ $\d^*\psi\in H^{1}(M\setminus\Gamma;2\pi\Z)$.
		
		\medskip

		Now, we set
		\begin{equation*}
			u_\circ(x)\coloneqq \exp\left(i\int_{\gamma(x_\circ,x)}\d^*\psi\right)
		\end{equation*}
		where $\gamma(x_\circ,x)$ is any curve in $M\setminus\Gamma$ connecting two points $x_\circ,x\in M\setminus\Gamma$. The integrality condition ensures that $u_\circ$ is independent on the choice of $\gamma$ and
		\begin{equation*}
			ju_\circ=(\d u_\circ,iu_\circ)=\d^*\psi.
		\end{equation*} 
		Let now $v$ be another map in $\mathscr{H}(M,\llb\Gamma\rrb)$. Then $j(u_\circ v^{-1})=ju_\circ-jv=\d^*\psi-\d^*\psi=0$ since there is no harmonic part. This implies that $u_\circ v^{-1}$ is a constant phase, so the two maps are equal in $\mathscr{H}(M,\llb\Gamma\rrb)$.
		
		\medskip
		
		If $b_1(M)\ne 0$ the proof is similar, with the difference that $[\d^*\psi]$ is in general not integral. Choosing $\omega_\alpha= [\d^*\psi]+\alpha$ with $\alpha\in H^1(M,2\pi\Z)$ is equivalent to saying that $[\d^*\psi+\omega_\alpha]\in H^1(M,2\pi\Z)$, so the map 
		\begin{equation*}
			u_\alpha(x)\coloneqq \exp\left(i\int_{\gamma(x_\circ,x)}\d^*\psi+\omega_\alpha\right)
		\end{equation*}
		is well defined and satisfies \eqref{eq: equalpha th1}. Let now $v$ be any other map in $\mathscr{H}(M,\llb\Gamma\rrb)$ with decomposition $jv=\d^*\psi+\omega_v$. Then 
		\begin{equation*}
			j(u_\alpha v^{-1})=\omega_\alpha-\omega_v
		\end{equation*} 
		which implies that $\omega_\alpha-\omega_v$ must be integral. In particular, 
		\begin{equation*}
			\omega_v\in \omega_\alpha+H^1(M,2\pi\Z)=[\d^*\psi]+H^1(M,2\pi\Z)
		\end{equation*}
		which concludes the proof.
	\end{proof}
\end{theorem}
\begin{definition}
	We define the desingularised energy of the harmonic map $v\in \mathscr{H}(M,\llb\Gamma\rrb)$ in the following way: if $jv$ decomposes as $jv=\d^*\psi+\omega$, then we set 
	\begin{equation*}
		e(v)\coloneqq \|\omega\|_{L^2(M)}^2.
	\end{equation*}
\end{definition}
\begin{remark}\label{rem: canonical}
	Theorem \ref{thm: harmonic maps} is sharp in the sense that in general, when $b_1(M)\ne 0$, $[\d^*\psi]$ is actually not integral and the lattice defined by 
	\begin{equation}\label{eq: lattice}
		[\d^*\psi]+H^1(M,2\pi\Z)
	\end{equation}
	is affine. 
	This was proved in a case as simple as the flat torus in \cite[Example 6.7]{Ignat-Jerrard2021}. The natural question about the possibility of choosing a \emph{canonical} singular harmonic map $u\in\mathscr{H}(M,\llb\Gamma\rrb)$ has a negative answer in general, since even the least energetic point (in the sense of the desingularised energy $e$) of the lattice \eqref{eq: lattice} might not be unique (see the same example in \cite{Ignat-Jerrard2021}).  
	\end{remark}
Before moving to the next section, we present a few properties of $\mathscr{H}(M,\llb\Gamma\rrb)$. Recall that if $\Gamma$ is embedded and $C^{1,1}$, we can parametrise a tubular neighbourhood via Fermi coordinates: there exists $\delta_\circ>0$ sufficiently small such that for any $\delta\in(0,\delta_\circ]$ the map
\begin{align*}
	X&\colon \Gamma\times B_\delta\to T_\delta\\
	&\ \ (y,z)\to \exp_y(z^1\nu_1(y)+z^2\nu_2(y))
\end{align*}
defines a diffeomorphism (see \cite{Gray2004}). Here $B_\delta$ denotes the ball of radius $\delta$ in $\R^2$ and $\{\nu_1,\nu_2\}$ is a global normal frame for $\Gamma$, oriented positively with $\vec\Gamma$.
\begin{lemma}\label{lem: propertis Hscr}
	The following properties of singular harmonic maps hold:
	\begin{enumerate}
		\item If $\llb\Gamma_1\rrb$ and $\llb\Gamma_2\rrb$ are admissible currents and $u_i\in \mathscr{H}(M,\llb\Gamma_i\rrb)$ for $i=1,2$, then $u_1u_2\in\mathscr{H}(M,\llb\Gamma_1\rrb+\llb\Gamma_2\rrb)$.
		\item If $u_1,\dots,u_m\in \mathscr{H}(M,\llb\Gamma\rrb)$, then $u_1\cdots u_m\in \mathscr{H}(M,m\llb\Gamma\rrb)$ where $m\llb\Gamma\rrb$ is counted with multiplicity $m$.
		\item If $u\in \mathscr{H}(M,\llb\Gamma\rrb)$, then $u^{-1}\in \mathscr{H}(M,-\llb\Gamma\rrb)$, where $-\llb\Gamma\rrb$ denotes the same current as $\llb\Gamma\rrb$ but with opposite orientation.
		\item If $u_\circ\in \mathscr{H}(M,\llb\Gamma\rrb)$ and $\Gamma$ is a $C^{1,1}$, embedded submanifold, then there exists a function $\lambda\in W^{1,q}(M,\R)$ for every $q\in(1,+\infty)$ such that
		\begin{equation*}
			u\vert_{T_\delta}=e^{i\lambda(y,z)}\frac{z}{|z|},
		\end{equation*}
		where $(y,z)$ are Fermi coordinates around $\Gamma$.
	\end{enumerate}
	\begin{proof}
		Properties $(1)$---$(3)$ follow directly from \eqref{eq: can harm map} and the group structure of $H^1(M,2\pi\Z)$. To show property $(4)$, we use that by the degree +1 condition and the admissibility condition there exists a map $u_*\colon M\to \s^1$, smooth outside of $\Gamma$ and such that $u_*\vert_{T_\delta}=z/|z|$ for any choice of normal frame in Fermi coordinates; see \cite[Appendix A]{Badran-DelPino2024}. In particular, the Hodge decomposition of $ju_*$ reads $ju_*=\d\lambda+\d^*\psi+\omega$, where $\psi=\psi(\llb\Gamma\rrb)$ and $\omega\in [\d^*\psi]+H^1(M,2\pi\Z)$. Multiplying $u_*$ with elements of the form $\exp(i\int_{\gamma(x_\circ,x)}\alpha)$, where $\alpha\in H^1(M,2\pi\Z)$, we can assume that $\omega=\omega_\circ$ where $\omega_\circ\in[\d^*\psi]+ H^1(M,2\pi\Z)$ is the only element for which $ju_\circ=\d^*\psi+\omega_\circ$. It remains to prove that, for every $q\in(1,+\infty)$, $\lambda$ is $W^{1,q}$ near $\Gamma$. Observe that $\lambda$ solves
		\begin{equation*}
			-\Delta_M\lambda=\d^*ju_*.
		\end{equation*}
		Near $\Gamma$,  $ju_*(y,z)=|z|^{-2}(-z_2\d z_1+z_1\d z_2)$ in Fermi coordinates and
		\begin{equation}\label{eq: expansion djustar}
			\d^*ju_*(y,z)=-H^1_\Gamma(y)\frac{z_2}{|z|^2}+ H^2_\Gamma(y)\frac{z_1}{|z|^2}  +V(y,z)
		\end{equation}
		where $V$ is bounded and $H_\Gamma^i=\langle H_\Gamma,\nu_i\rangle$, with $H_\Gamma$ being the mean curvature vector of $\Gamma$. Indeed, this is a direct computation using that $\d^*\omega=-\frac{1}{\sqrt{g}}\partial_i(\sqrt{g}g^{ij}\omega_j)$ and that the metric in Fermi coordinates is given by 
		\begin{equation}\label{eq: Fermi metric}
			g=\begin{pmatrix}
				g_{z}&\O(|z|^2)\\\O(|z|^2)&I_2+\O(|z|^2)
			\end{pmatrix}
		\end{equation}
		where $g_z$ is a two-parameter family of $C^{1,1}$ metrics on $\Gamma$ and $I_2$ is the $2\times 2$ identity matrix. Note that, up to very small terms the metric \eqref{eq: Fermi metric} is diagonal near $\Gamma$, which means that we can approximate locally $T_\delta$ with $\Gamma\times B_\delta$, with a quadratically small error in $|z|$. Regularity properties for the Laplace--Beltrami operator $\Delta_{T_\delta}$ on the tubular neighbourhood then follow from regularity properties of the product operator $\Delta_{\Gamma\times B_\delta}=\Delta_\Gamma+\Delta_z$ via a perturbative approach \emph{\`a la} Schauder, see \cite[\S2.3]{FernandezReal-RosOton}. The advantage of working in a tube $\Gamma\times B_\delta$ resides in the possibility of ``splitting'' the regularities along and across the submanifold $\Gamma$, in the following way: the right-hand side $\d^*ju_*$ belongs to the anisotropic Sobolev space $L^q_yL^p_z(\Gamma\times B_\delta)$, defined by the norm 
		\begin{equation*}
			\|f\|_{L^q_yL^p_z(\Gamma\times B_\delta)}=\left(\int_{\Gamma}\left(\int_{B_\delta}|f|^p\right)^{q/p}\right)^{1/q}
		\end{equation*}
		for every $p\in(1,2)$ and $q\in(1,+\infty)$. Anisotropic Sobolev spaces enjoy a Calder\'on--Zygmund theory \cite{Stefanov-Torres2004}, which guarantees that $\lambda\in L^q_yW^{2,p}_z(\Gamma\times B_\delta)$. In turn, since the $W^{2,p}$ norm acts on functions defined on 2-dimensional domains, by Sobolev embedding
		\begin{equation*}
			L^q_yW^{2,p}_z(\Gamma\times B_\delta)\subset L^q_yW^{1,q}_z(\Gamma\times B_\delta).
		\end{equation*} 
		More precisely, this means that 
		\begin{equation}\label{eq: LqW1qz}
			\left(\int_{\Gamma\times B}|D_z\lambda|^q+|\lambda|^q\right)^{1/q}<\infty.
		\end{equation}
		Using the fact that the right-hand side is Lipschitz in $y$, we can differentiate the whole equation in $y$ to obtain $-\Delta(D_y\lambda)=D_yf\in L^q_yL^p_z(\Gamma\times B_\delta)$. Using again Calder\'on--Zygmund and Sobolev embedding, we find that $D_y\lambda\in L^q(\Gamma\times B_\delta)$ and, combining with \eqref{eq: LqW1qz}, we get $\lambda\in W^{1,q}(\Gamma\times B_\delta)$ for any $q\in (1,+\infty)$. This implies in particular that $\lambda\in C^{0,\alpha}(\Gamma\times B_\delta)$. 
		\end{proof}
\end{lemma}

\subsection{The renormalised energy}\label{subs: ren en}
From now on, we specialise to the case where $\Gamma$ is embedded and $C^{1,1}$. First, we fix a map $u_\circ\in\mathscr{H}(M,\llb\Gamma\rrb)$, with decomposition 
\begin{equation}\label{eq: decomp ucirc}
	ju_\circ=\d^*\psi(\llb\Gamma\rrb)+\omega_\circ.
\end{equation}
Then, let $\delta>0$ sufficiently small and let $M_\delta\coloneqq M\setminus T_\delta$, where $T_\delta$ is a $\delta$-tubular neighbourhood of $\Gamma$. Then the energy of $u_\circ$ on $M_\delta$ is finite and we investigate its asymptotic expansion as $\delta\to 0$ (see Proposition \ref{prop: expansion can harm map}), where $W_M(\Gamma)$ will make its appearance. 

\medskip

The renormalised energy arises from the double integral 
\begin{equation}\label{eq: double interaction}
	4\pi^2\int_{\Gamma}\int_{\Gamma}
	(G_M(x,y),\vec\Gamma(x)\wedge\vec\Gamma(y))_g,
\end{equation}
encoding a self-interaction of the singular set mediated by the fundamental solution $G_M\coloneqq G_M^{n-2}$. Observe that \eqref{eq: double interaction} is not finite due to non-summability near the diagonal $\{x=y\}$; the renormalised energy $W_M(\Gamma)$ is defined exactly as the desingularisation of \eqref{eq: double interaction}. To do so, we use the following result.

\begin{lemma}\label{lem: decomposition w}
	Let $\llb\Gamma\rrb$ be admissible, with embedded support and $C^{1,1}$. Let $A$ be the solution to $\Delta A=2\pi\llb\Gamma\rrb$ as in the proof of Lemma \ref{lem: solvability psi}.
	Denote with $(y,z)$ the Fermi coordinates in $T_{\delta_0}$ and let $S\in D_{n-2}(M)$ be the $(n-2)$-current defined as 
	\begin{equation*}
		S=\begin{cases}
			-\xi(|z|)\log|z|\vec\Gamma(y)& \text{in }\ T_{\delta_0}\\
			0& \text{in }\ M\setminus T_{\delta_0}
		\end{cases}
	\end{equation*}
	where $\xi$ is a smooth cut-off function such that $\xi\equiv 1$ in $[0,\delta_0/2)$ and $\xi\equiv 0$ in $[\delta_0,+\infty)$
	Then, the current
	\begin{equation*}
		R\coloneqq A+S\in D_{n-2}(M)
	\end{equation*}
	is of class $W^{1,q}(M)$ for every $q\in (1,+\infty)$. In particular, $R$ is H\"older continuous.
\end{lemma}
We postpone the proof of Lemma \ref{lem: decomposition w} to \S\ref{subs: Proof energy expansion} and describe the renormalised energy.
The interaction integral \eqref{eq: double interaction} can be expressed as 
\begin{align*}
	4\pi^2\int_{\Gamma}\int_{\Gamma}
	(G_M(x,y),\vec\Gamma(x)\wedge\vec\Gamma(y))_g&=4\pi^2\int_\Gamma\left(\int_\Gamma\langle G_M(\cdot,y),\vec\Gamma(y)\rangle_y\right)\\
	&=2\pi\int_\Gamma\int_M\langle G_M(\cdot,y),2\pi\llb\Gamma(y)\rrb\rangle_y\\
	&=2\pi\int_\Gamma A
\end{align*}
We define the renormalised energy as the (negative) desingularisation of this integral, obtained by replacing $A$ with its regular part, namely 
\begin{equation}\label{eq: ren en obscure}
	W_{M}(\Gamma)\coloneqq -2\pi\int_{\Gamma}R.
\end{equation}
Now, we can state the next result.
\begin{proposition}\label{prop: expansion can harm map}
	Let $u_\circ\in\mathscr{H}(M,\llb\Gamma\rrb)$ satisfy \eqref{eq: decomp ucirc}. Then, the energy expansion 
	\begin{equation}\label{eq: expansion ucirc}
		\int_{M_\delta}|\d u_\circ|^2=2\pi\mathcal{H}^{n-2}(\Gamma)\log(1/\delta)+W_M(\Gamma)+e(u_\circ)+\O(\delta^\beta)
	\end{equation}
	holds as $\delta\to 0^+$, where $e(u_\circ)=\int_{M}|\omega_\circ|^2$ and $\beta\in(0,1).$
\end{proposition}
Expression \eqref{eq: ren en obscure} for the renormalised energy is perhaps a bit obscure. Note that to understand the interaction between \emph{different parts} of the singular set, expression \eqref{eq: double interaction} remains valid, as the desingularisation happens just locally near the diagonal, see Appendix \ref{app: ren en Rn}. On the other hand, the singular part of $A$ is what contributes to the first (blowing-up) term in the energy expansion \eqref{eq: expansion ucirc}.

\subsection{Proof of Proposition \ref{prop: expansion can harm map}}\label{subs: Proof energy expansion}
To prove the Proposition we need two preliminary lemmas. The first one is Lemma \ref{lem: decomposition w}, which we prove next.
\begin{proof}[Proof of Lemma \ref{lem: decomposition w}]
		By the Weitzenb\"ock formula, the Hodge Laplacian on a $k$-form $\omega$ is given by 
		\begin{equation*}
			\Delta \omega=\nabla^*\nabla \omega+\Ric(\omega)
		\end{equation*}
		where $\nabla^*\nabla$ is the Bochner Laplacian and $\Ric(\omega)$ is the Weitzenb\"ock curvature of $\omega$ (see \cite{Petersen2006}). The Bochner Laplacian has a metric representation given by 
		\begin{equation*}
			\nabla^*\nabla =-g^{jk}\nabla_j\nabla_k-\frac{1}{\sqrt{|g|}}\partial_j(\sqrt{|g|}g^{jk})\nabla_k
		\end{equation*}
		where $\nabla$ is the Levi-Civita connection on $(M,g)$. We are interested in computing the above quantity on a form given by $\omega=f(z)\eta(y)$, where $f$ is a scalar function and $\eta$ is a smooth form on $\Gamma$, being $(y,z)$ Fermi coordinates. 
		
		\medskip
		
		By the properties of the Levi-Civita connection, we have
		\begin{align*}
			\nabla_k(f\eta)&=f\nabla_k\eta+\partial_kf\eta \\
			\nabla_j\nabla_k(f\eta)&=f\nabla_j\nabla_k\eta+\partial_jf\nabla_k\eta+\partial_kf\nabla_j\eta+\partial_{jk}f\eta
		\end{align*}
		which implies
		\begin{equation*}
			\nabla^*\nabla(f\eta)=f\nabla^*\nabla\eta+2g^{jk}\partial_jf\nabla_k\eta-(\Delta_Mf)\eta.
		\end{equation*}
		Applying these formulas to our case, namely $f(z)=\log|z|$ and $\eta(y)=\vec\Gamma(y)$, and using that on a function $f(z)$ of the normal coordinate $z$ the Laplace--Beltrami operator is
		\begin{equation*}
			\Delta_Mf(z)=\Delta_zf-(H_\Gamma(y)+\O(|z|^2))\cdot\nabla_zf,
		\end{equation*}
		we get
		\begin{equation*}
			-(\Delta_z\log|z|)\vec\Gamma(y)=2\pi\llbracket\Gamma\rrbracket-T
		\end{equation*}
		where $T$ is a sum of terms that are either regular or that can be expressed as $\xi(y)\log|z|$ or $\zeta(y)\nabla_z\log|z|$, for some Lipschitz forms $\xi,\zeta$ on $\Gamma$. Then, $R$ solves \begin{equation*}
			\Delta R=\Delta A+\Delta S=T.
		\end{equation*}
		As in point (4) of Lemma \ref{lem: propertis Hscr}, $W^{1,q}$ regularity of $R$ follows by the fact that the right hand side belongs to $L^q_yL^p_z(\Gamma\times B_\delta)$, with $(p,q)\in(1,2)\times(2,+\infty)$, on an arbitrarily small tubular neighbourhood of $\Gamma$, applying Calder\'on--Zygmund estimates and Sobolev inequality.
		\end{proof}

\begin{corollary}\label{cor: estimates A}
	For every $r\in [1,2)$ there exists $C>0$ and $\beta\in(0,1)$ such that 
	\begin{align*}
		\|\d A\|_{L^r(T_\delta)}&\leq C\delta^\beta.
	\end{align*}
	\begin{proof}
		By Lemma \ref{lem: decomposition w}, $\d R$ is $L^q$ for every $q>2$ and its $L^r(T_\delta)$ norm is bounded by $\|\d R\|_{L^q}\operatorname{Vol}(T_\delta)^{1/r-1/q}=\O(\delta^\frac{2(q-r)}{rq})$. The term $\d S=-\d(\log|z|)\wedge\vec\Gamma(y)$ can be integrated in Fermi coordinates and the integral is controlled by 
		\begin{equation*}
			\int_{T_\delta}|\d S|^r\leq C(\Gamma,M)\int_{B_\delta\subset\R^2}\frac{\d z}{|z|^r}\leq C(r)\delta^{2-r}.
		\end{equation*}
	\end{proof}
\end{corollary}

\begin{lemma}\label{lem: integration of forms}
	If $\alpha\in D_{n-2}(M)$, then 
	\begin{equation}\label{stokes}
	\int_{\partial T_\delta}\alpha\wedge\star \d A=\int_{T_\delta}( \d\alpha, \d A)_g-2\pi \int_{\Gamma}\alpha
\end{equation}
\begin{proof}
	By direct differentiation we know that
	\begin{align*}
		\d(\alpha\wedge \star \d A)&=\d\alpha\wedge \star \d A+(-1)^n\alpha \wedge \d\star \d A \\
		&=(\d\alpha,\d A)_g +(-1)^n(\alpha,\star \d\star \d A)_g.
	\end{align*}
	Recall that, by \eqref{eq: dd* on w}, $\Delta A=\d^*\d A$ and that 
	\begin{equation*}
		\d^*=(-1)^{n(p+1)+1}\star \d\star
	\end{equation*}
	whenever $\d^*$ acts on $p$-forms on an $n$-dimensional manifold. In our case, $p=n-1$ and $(-1)^{n^2+1}=(-1)^{n+1}$, therefore $(-1)^n\star \d\star \d A=-\Delta A$ and
	\begin{equation*}
		\d(\alpha\wedge \star \d A)=(\d\alpha,\d A)_g -(\alpha,\Delta A)_g.
	\end{equation*}
	Integrating over $T_\delta$ and using Stokes' theorem yields the result, recalling also that $\Delta A =2\pi \llbracket \Gamma\rrbracket$.
\end{proof}
\end{lemma}
Next, we compute
\begin{equation}\label{expansion energy gruyere}
	\begin{split}
		\int_{M_\delta}|\d u_\circ|^2&=\int_{M_\delta}|ju_\circ|^2\\&=\int_{M_\delta}|\d^*\psi|^2+\int_{M_\delta}|\omega_\circ|^2+2\int_{M_\delta}(\d^*\psi,\omega_\circ)_g.
	\end{split}
\end{equation}
By boundedness of $|\omega_\circ|$, 
\begin{equation*}
	\int_{M_\delta}|\omega_\circ|^2=\int_{M}|\omega_\circ|^2+\O(\delta^2)
\end{equation*}
and by Corollary \ref{cor: estimates A} we have 
\begin{equation*}
	\left|\int_{M_\delta}(\d^*\psi,\omega_\circ)_g\right|=\left|\int_{T_\delta}(\d^*\psi,\omega_\circ)_g \right|\leq C \|\omega_\circ\|_{\infty}\|\d A\|_{L^1(T_\delta)}\leq C\delta
\end{equation*}
where the first equality follows by the fact that $\int_M(\d^*\psi,\omega_\circ)_g=\int_M(\psi,\d\omega_\circ)_g=0$.

\medskip

The squared $L^2$ norm of $\d^*\psi$ is the term that will yield the area and renormalised energy contributions in the energy expansion. By Stokes' theorem 
\begin{equation}\label{eq: square norm singular term}
	\int_{M_\delta}|\d^*\psi|^2=\int_{M_\delta}|\d A|^2=\int_{\partial M_\delta}A\wedge\star \d A
\end{equation}
where we used that the support of $\Delta A$ is $\Gamma$.
To compute \eqref{eq: square norm singular term}, we use Lemma \ref{lem: decomposition w} and Lemma \ref{lem: integration of forms}. We have 
\begin{equation}\label{eq: decomposition}
	\begin{split}
		\int_{\partial M_\delta}A\wedge\star \d A&=\int_{\partial T_\delta}A\wedge\star \d A\\
		&=-\int_{\partial T_\delta}S\wedge\star \d A+\int_{\partial T_\delta}R\wedge\star \d A\\
		&=2\pi\log(1/\delta)\int_{\Gamma}\vec\Gamma-2\pi\int_{\Gamma}R+2\pi\int_{T_\delta}( \d R,\d A)_g\\
		&=2\pi\mathcal{H}^{n-2}(\Gamma)\log(1/\delta)+W_{M}(\Gamma)+2\pi\int_{T_\delta}( \d R,\d A)_g
	\end{split}
\end{equation}
where we used that $\d\vec\Gamma=0$ and that
\begin{align}
	\int_{\Gamma}\vec\Gamma&=\int_{\Gamma}(\vec\Gamma(y),\vec\Gamma(y))_g d\mathcal{H}^{n-2}(y)=\mathcal{H}^{n-2}(\Gamma).\label{eq: Gamma_j}
\end{align}
Finally, by Corollary \ref{cor: estimates A}, if $r\in (1,2)$ and $1/r+1/q=1$
\begin{equation*}
	\int_{T_\delta}(\d R,\d A)_g \leq C \|\d R\|_{q}\|\d A\|_{L^r(T_\delta)}\leq C\delta^\beta
\end{equation*}
and the proof is concluded.\qed

\section{Existence of minimisers and relaxed energies expansion}\label{sec: minimisers}
The second part of this work investigates three different variational problem approximating the Dirichlet energy with prescribed singular set. The first one is obtained by removing a $\delta$-tubular neighbourhood of $\Gamma$ and minimising among all the $W^{1,2}$ maps with degree +1 around $\Gamma$, in the spirit of \cite{Bethuel-Brezis-Helein1994}. By a direct comparison, together with some regularity near the boundary obtained with a scheme \`a la De Giorgi, we show that the energies of minimisers have an expansion similar to \eqref{eq: expansion ucirc} for $u_\circ$.

\medskip

Secondly, we propose $p$-harmonic and $s$-harmonic energy relaxations, obtained by replacing the critical norm $W^{1,2}$ with the subcritical norms $W^{1,p}$ ($p<2$) and $W^{s,2}$ ($s<1$), respectively. Here, we show that the energy expansion features again the same renormalised energy $W_M(\Gamma)$, revealing the self-interaction of the singular set visible only in the diffuse regime ($p<2$ or $s<1$).

\subsection{Relaxation by removing $\delta$-tubes}
Given $\delta>0$ let $T_\delta$ be a $\delta$-neighbourhood of $\Gamma$ and $M_\delta\coloneqq M\setminus T_\delta$. Consider the energy 
\begin{equation*}
	E_\delta(u)\coloneqq \int_{M_\delta}|\d u|^2.
\end{equation*}
In order to recover the singular set in the limit $\delta\to 0$ we prescribe the degree $\deg(u,\Gamma)$ of $u$ around $\Gamma$ to be +1 (or the desired multiplicity), namely we minimise 
\begin{equation}\label{eq: delta minimisation}
	\min_{u\in W^{1,2}_\Gamma(M_\delta,\s^1)} E_\delta(u)
\end{equation}
where
\begin{equation}\label{eq: W12 with deg 1}
	W^{1,2}_\Gamma(M_\delta,\s^1)\coloneqq \{u\in W^{1,2}(M_\delta,\s^1):\deg(u,\Gamma)=1\}.
\end{equation}
We can give a better description of the set \eqref{eq: W12 with deg 1}. Consider the quotient 
\begin{equation*}
	W^*_{\Gamma}\coloneqq W^{1,2}_\Gamma(M_\delta,\s^1)/\sim
\end{equation*}
where two elements are identified if their ratio is a $W^{1,2}$ phase, namely 
\begin{equation*}
	u\sim v\iff u=e^{i\varphi}v,\quad \varphi\in W^{1,2}(M_\delta,\R).
\end{equation*}
\begin{proposition}\label{prop: topology decomposition}
	Each class in $W^*_{\Gamma}$ can be represented by exactly one class $u_\alpha\in\mathscr{H}(M,\llb\Gamma\rrb)$. In other words, there is a bijection 
	\begin{equation*}
		W^*_{\Gamma}\to \mathscr{H}(M,\llb\Gamma\rrb).
	\end{equation*}
	\begin{proof}
		We start by describing the decomposition of $ju$, where $u$ is any map in $W^{1,2}_\Gamma(M_\delta,\s^1)$. By $W^{1,2}$ integrability, we have $Ju=0$ and in particular 
		\begin{equation}\label{eq: decomp prop}
			ju=\d\varphi+\omega
		\end{equation}
		where $\varphi\in W^{1,2}(M_\delta,\R)$ and $\omega\in\mathcal{H}^1_\Delta(M_\delta)$. Our claim is that, thanks to the admissibility condition, the global topology of $M_\delta$ ``splits'' into the topology of $M$ and the topology around the cavity $T_\delta$.
		
		\medskip
		
		Consider the long exact sequence \eqref{eq: les} with de Rham cohomologies ($G=\R$) 
		\begin{equation*}
			\begin{tikzcd}
				 H^1(M,M\setminus\Gamma) \arrow[r, "j^*"]  & H^1(M) \arrow[r, "i^*"] & H^1(M\setminus\Gamma)\arrow[r,"\d"] & H^2(M,M\setminus\Gamma) \arrow[r, "j^*"]  & H^2(M)
			\end{tikzcd}
		\end{equation*}
		As we did in the proof of Theorem \ref{thm: harmonic maps}, we can argue that $j^*(H^k(M,M\setminus\Gamma))=0$ for $k=1,2$, thanks to the fact that $\Gamma$ is a boundary. This implies that the short sequence 
		\begin{equation*}
			\begin{tikzcd}
				0 \arrow[r]  & H^1(M) \arrow[r, "i^*"] & H^1(M\setminus\Gamma)\arrow[r,"\d"] & H^2(M,M\setminus\Gamma) \arrow[r]  & 0
			\end{tikzcd}
		\end{equation*}
		is exact, and in particular that every $\omega\in H^1(M\setminus\Gamma)$ can be split as $\omega=i^*\omega_1+s\omega_2$, where $\omega_1\in H^1(M)$, $\omega_2\in H^2(M,M\setminus\Gamma)$ and $s$ is a right-inverse of $\d$. 
		
		\medskip
		
		With this fact and by Hodge theorem, equation \eqref{eq: decomp prop} can be written as 
		\begin{equation*}
			ju=\d\varphi+\omega_1+\omega_2
		\end{equation*}
		where $\omega_1\in \mathcal{H}_\Delta^1(M)$ and $\omega_2\in \mathcal{H}_\Delta^1(M\setminus\Gamma)$ satisfying $\omega_2=s\widetilde \omega$, with $\widetilde\omega\in H^2(M,M\setminus\Gamma)$. Using the condition $\deg(u,\Gamma)=1$ and the properties of the distributional degree \cite[\S12.4]{Brezis-Mironescu2021}, one sees that $[\omega_2]=[\d^*\psi]$ where $\psi=\psi(\llb\Gamma\rrb)$ is the one arising from Lemma \ref{lem: decomposition w}. 
		
		\medskip
		
		Lastly, we remark that in the same class of $u$ in $W^*_\Gamma$ lies the map $\widetilde u=e^{-i\varphi}u$ whose decomposition is 
		\begin{equation}\label{eq: decomp utilde}
			j\widetilde u=\d^*\psi+\omega_1.
		\end{equation}
		In particular, $\omega_1\in [\d^*\psi]+H^1(M,2\pi\Z)$ and thus $\widetilde u\in \mathscr{H}(M,\llb\Gamma\rrb)$. Clearly, $\widetilde u$ is the unique map (up to constant phases) in the given sector satisfying \eqref{eq: decomp utilde}. This concludes the proof. 
	\end{proof}
\end{proposition}

By Proposition \ref{prop: topology decomposition}, we can rephrase \eqref{eq: delta minimisation} as a double minimisation problem 
\begin{equation}\label{eq: double minimisation}
	\min_{u\in W^{1,2}_\Gamma(M_\delta,\s^1)} E_\delta(u)=\min_{u_\alpha\in \mathscr{H}(M,\llb\Gamma\rrb)}\min_{u\in W^{1,2}_\alpha(M_\delta,\s^1)}E_\delta(u)
\end{equation}
where 
\begin{equation*}
	W^{1,2}_\alpha(M_\delta,\s^1)\coloneqq \left\lbrace e^{-i\varphi}u_\alpha: \varphi\in W^{1,2}_\alpha(M_\delta,\R)\right\rbrace.
\end{equation*}
This constraint reduces the problem to a scalar one. Indeed, for $u\in W^{1,2}_\alpha(M_\delta,\s^1)$, note that $|\d u|=|ju|=|\d\varphi-ju_\alpha|$, hence the inner minimisation problem in \eqref{eq: double minimisation} becomes equivalent to 
\begin{equation}\label{eq: minimisation delta with psi}
	\min_{\varphi\in W^{1,2}(M_\delta,\R)}\int_{M_\delta}|\d \varphi-ju_\alpha|^2.
\end{equation}
Observe that the minimiser $\varphi_\delta$ of \eqref{eq: minimisation delta with psi} exists unique (up to a constant) by convexity of the energy. Also, $\varphi_\delta$ is harmonic in $M_\delta$ since it satisfies weakly
\begin{equation*}
	-\Delta\psi_\delta=\d^*\d\psi_\delta=\d^*ju_\alpha=0\quad\text{in }M_\delta
\end{equation*}
plus the natural boundary condition 
\begin{equation*}
	\partial_\nu\varphi_\delta=\langle ju_\alpha,\nu\rangle\quad\text{on }\partial M_\delta.
\end{equation*}
\begin{remark}
	If there exists an optimal harmonic map $u_\circ\in \mathscr{H}(M,\llb\Gamma\rrb)$, namely a strict minimiser of 
	\begin{equation*}
		\inf_{u\in\mathscr{H}(M,\llb\Gamma\rrb)}e(u),
	\end{equation*}
	then the minimisation over $W_\Gamma^{1,2}(M_\delta,\s^1)$ will pick the sector corresponding to the optimal map $u_\circ$, at least for $\delta$ sufficiently small. This is a consequence of the energy expansion of Theorem \ref{thm: expansion udelta}. 
\end{remark}

\subsection{$p$-harmonic maps to $\s^1$}\label{subs: pharm}
Let $p\in (1,2)$. As a relaxation of the Dirichlet energy, we consider 
\begin{equation}\label{eq: p energy}
	E_p(u)=\int_{M}|\d u|^p.
\end{equation}
Energy \eqref{eq: p energy} is well defined on maps with a singularities as the ones in $\mathscr{H}(M,\llb\Gamma\rrb)$ since the singularity is $W^{1,p}$. For this reason, we can minimise in the entire space and rather than the degree we can prescribe directly the distributional Jacobian (which makes sense, as in \eqref{eq: dist jac}, for maps in $W^{1,p}$). More precisely, we consider 
\begin{equation}\label{eq: p minim}
	\min_{u\in W^{1,p}_\Gamma(M,\s^1)} E_p(u)
\end{equation} 
where 
\begin{equation*}
	W^{1,p}_\Gamma(M,\s^1)\coloneqq \left\lbrace u\in W^{1,p}(M,\s^1) : Ju=2\pi \star\llb\Gamma\rrb\right\rbrace.
\end{equation*}
As before, we can rephrase \eqref{eq: p minim} as a double minimisation problem
\begin{equation*}
	\min_{u\in W^{1,p}_\Gamma(M,\s^1)} E_p(u)=\min_{u_\alpha\in \mathscr{H}(M,\llb\Gamma\rrb)}\min_{u\in W^{1,p}_\alpha(M,\s^1)}E_p(u)
\end{equation*}
where
\begin{equation*}
	W^{1,p}_\alpha(M,\s^1)\coloneqq \{u=e^{-i\varphi}u_\alpha:\varphi\in W^{1,p}(M,\R)\}.
\end{equation*}
Indeed, the condition $Ju=2\pi \star\llb\Gamma\rrb$ forces $\psi=\psi(\llb\Gamma\rrb)$ in the Hodge decomposition $ju=\d\varphi+\d^*\psi+\omega$. As above, we can consider equivalence classes modulo $W^{1,p}$ phases and find that, for each of these classes, there is a representative in $\mathscr{H}(M,\llb\Gamma\rrb)$.

 The problem of minimising $E_p$ in $W^{1,p}_\alpha(M,\s^1)$ admits a unique (up to a constant phase) solution $u_p=e^{-i\varphi_p}u_\alpha$. Existence follows as in $E_\delta$, noting that the Lagrangian $|\d\varphi-ju_\circ|^p$ is convex. Lastly observe that, by \eqref{eq: algebraic properties J},
\begin{equation}\label{eq: phase does not change singular set}
	Ju_p= J(e^{i\varphi_p})+Ju_\circ=Ju_\circ=2\pi \star \llbracket\Gamma\rrbracket,
\end{equation}
so the singular set of $u_p$ is still $\llb\Gamma\rrb$. 

\subsection{Fractional harmonic maps to $\s^1$}
Let $s\in (\tfrac12,1)$ and consider the energy 
\begin{equation}\label{eq: fractional energy}
	E_s(u)=[u]_{H^{s}(M)}^2
\end{equation}
among all maps $u\colon M\to\s^1$ such that the Sobolev seminorm 
\begin{equation*}
	[u]_{H^{s}(M)}^2\coloneqq\iint_{M\times M}|u(x)-u(y)|^2K_s(x,y)dxdy
\end{equation*}
is finite. Fractional Sobolev spaces on manifolds were recently introduced in \cite{Caselli-Florit-Serra2024}. The fractional kenrel $K_s$ is obtained by the heat kernel $H_M$ of $M$ via 
\begin{equation*}
	K_s(x,y)=\frac{s}{\Gamma(1-s)}\int_0^\infty H_M(x,y,t)\frac{dt}{t^{1+s}}
\end{equation*}
The renormalisation constant ${s}/{\Gamma(1-s)}$ is guarantees that for any $v\in H^1(M)$
\begin{equation*}
	E_s(v)\to \int_{M}|\d v|^2
\end{equation*}
as $s\to 1^-$.

\medskip

Here, we prescribe the singular set in the sense of Definition \ref{def: singular set Hs}, using the distributional Jacobian for $H^s$ functions. Recall again that maps $u\in H^s(M,\s^1)$ have a factorisation $u=e^{i\varphi}v$, where $\varphi\in H^s(M,\R)$ and $v\in W^{1,2s}(M,\R)$, and the distributional Jacobian of $u$ is defined as $\widetilde Ju\coloneqq Jv$. More precisely, we minimise
\begin{equation}\label{eq: s minim}
	\min_{u\in H^s_\Gamma(M,\s^1)} E_s(u)
\end{equation}
where $H^s_\Gamma(M,\s^1)=\{u\in H^s(M,\s^1):\widetilde Ju=2\pi\star\llb\Gamma\rrb\}$.
Factorisation tells us that if we consider the quotient space of $H^s$ maps modulo $H^s$ phases, each class will have a $W^{1,2s}$ representative. The same argument of \S\ref{subs: pharm} then applies and, once again, we find an equivalence of \eqref{eq: s minim} with a double minimisation problem
\begin{equation*}
	\min_{u\in H^s_\Gamma(M,\s^1)} E_s(u)=\min_{u_\alpha\in \mathscr{H}(M,\llb\Gamma\rrb)}\min_{u\in H^s_\alpha(M,\s^1)}E_p(u)
\end{equation*}
where
\begin{equation*}
	H^s_\alpha(M,\s^1)\coloneqq \{u=e^{-i\varphi}u_\alpha:\varphi\in H^s(M,\R)\}.
\end{equation*}
The map $\varphi\mapsto E_s(e^{i\varphi}u_\circ)$ is weakly lower-semicontinuous in the $H^s$ topology, which implies the existence of a minimiser $\varphi_s$ by direct method. Moreover, by the very definition of $\widetilde J$, we have 
\begin{equation*}
	\widetilde J(e^{-i\varphi_s}u_\alpha)=\widetilde Ju_\alpha=2\pi\star\llb\Gamma\rrb.
\end{equation*}

\begin{remark}\label{rem: Fubini}
	Functions in $W^{1,p}$ and $H^{s}$ are in principle not regular enough to define a ``degree''. However, one can still make sense of the notion of linking around $\Gamma$ by observing that such functions are regular on sets of the form $\{|z|=r\}$, where $z$ are the normal Fermi coordinates and $r$ is small. For instance, if $u\in W^{1,p}$,
	\begin{equation*}
		\int_\Gamma \d y \int_{0}^{\delta}\d r\int_{\{|z|=r\}}|\d u|^p\leq C(\Gamma,M,\delta)\int_{T_\delta}|\d u|^p<+\infty.
	\end{equation*}
	which implies that $u\vert_{\{|z|=r\}}\in W^{1,p}(\{|z|=r\},\s^1)$ for a.e.~ $r\in(0,\delta)$. In particular, $u\vert_{\{|z|=r\}}$ is continuous for a.e.~ $r\in(0,\delta)$ and $\deg(u|{\{|z|=r\}})$ makes sense. A similar argument works for $u\in H^s$.
\end{remark}

\subsection{Convergence to the expansion of $u_\delta$}

The goal of this section is to prove an energy expansion for the minimiser $u_\delta$ to 
\begin{equation}\label{eq: minimisation delta with u}
	\min_{u\in W^{1,2}_\circ(M_\delta,\s^1)}E_\delta(u)
\end{equation}
arising in \eqref{eq: double minimisation}, where $W^{1,2}_\circ(M_\delta,\s^1)=\{e^{-i\varphi}u_\circ:\varphi\in W^{1,2}(M_\delta,\R)\}$, being $u_\circ$ a fixed choice of harmonic map in $\mathscr{H}(M,\llb\Gamma\rrb)$. We establish that the energy of the minimiser of \eqref{eq: minimisation delta with u} coincides up to constant order with the energy \eqref{eq: expansion ucirc} of the harmonic map $u_\circ$ it was constructed from. 
\begin{theorem}\label{thm: expansion udelta}
	Let $u_\delta$ be the minimiser arising from \eqref{eq: minimisation delta with u}. Then, there is $\gamma\in(0,1)$ such that its energy satisfies 
	\begin{equation}\label{eq: expansion udelta}	E_\delta(u_\delta)=2\pi\mathcal{H}^{n-2}(\Gamma)\log(1/\delta)+W_M(\Gamma)+e(u_\circ)+\O(\delta^{\gamma})
	\end{equation}
	as $\delta\to 0$.
\end{theorem}
Before proving Theorem \ref{thm: expansion udelta} we show a H\"older regularity lemma for a class of right--hand sides with the same regularity as $\Delta_M\lambda$, where $\lambda$ is the function arising from Lemma \ref{lem: propertis Hscr}.
\begin{lemma}\label{lem: regularity decay}
	Let $0<\delta<\delta'$ be sufficiently small and let $\psi$ be a solution to 
	\begin{equation*}
		\begin{cases}
			-\Delta_M\psi=\widetilde f&\text{in }M_\delta\\
			\partial_\nu\psi=0 &\text{on }\partial M_\delta
		\end{cases}
	\end{equation*}
	where $\widetilde f\colon M\to \R$ is regular outside of $T_{\delta'}$ and inside $T_{\delta'}$, denoting with $(y,z)$ Fermi coordinates, $\widetilde f$ is Lipschitz in $y$ and $L^p$ in $z$, for some $p\in (1,2)$. Then there are constants $C_\circ>0$ and $\gamma\in (0,1)$, independent on $\delta$, such that for every $r\in[\delta,\delta')$ there exists a function $\alpha_{r,\delta}\in W^{1,\infty}(\Gamma)$ for which
	\begin{equation*}
		\|\psi-\alpha_{r,\delta}\|_{L^\infty(T_r\setminus T_\delta)}\leq C_\circ r^\gamma.
	\end{equation*}
	\begin{proof}
	We divide the proof in four steps.
	
	\medskip
	
	{\bf Step 1} (Reduction to two dimensions).
		As we did in Lemma \ref{lem: propertis Hscr}, we approximate $T_{\delta'}\setminus T_\delta$ with a product space $\Gamma\times A_{\delta',\delta}$, where $A_{\delta',\delta}\coloneqq B_{\delta'}\setminus B_\delta\subset\R^2$ is a planar annulus. With this identification, we know that $\widetilde f\in L^p_zL^q_y(\Gamma\times A_{\delta',\delta})$ for any $q\in (1,+\infty)$, which by Calder\'on--Zygmund implies that $\psi\in L^p_zW^{2,q}_y(\Gamma\times A_{\delta',\delta})\subset L^p_zW^{1,\infty}_y(\Gamma\times A_{\delta',\delta})$. Using the Lipschitz-in-$y$ property of the right-hand side, we can apply the same reasoning to $D_y\widetilde f$ and obtain that $D_y\psi\in L^p_zW^{1,\infty}_y(\Gamma\times A_{\delta',\delta})$, with norm bounded by $\|\widetilde f\|_{L^p_xW^{1,\infty}_y(\Gamma\times B_{\delta'})}$, which is independent on $\delta$. But then, 
		\begin{equation*}
			\|\Delta_y \psi\|_{L^p_zL^\infty_y}\leq \|D_y \psi\|_{L^p_zW^{1,\infty}_y}\leq \|\widetilde f\|_{L^p_zW^{1,\infty}_y}.
		\end{equation*}
		Remark that, by Minkowski's integral inequality 
		\begin{equation*}
			\|\Delta_y \psi\|_{L^\infty_yL^p_z}\leq \|\Delta_y \psi\|_{L^p_zL^\infty_y},
		\end{equation*}
		see for instance \cite{Grey-Sinnamon2016}.
		Then, we can denote $f\coloneqq \widetilde f+\Delta_y\psi\in L^\infty_yL^p_z(\Gamma\times A_{\delta',\delta})$ and note that $\psi_y\coloneqq\psi(\cdot,y)$ solves
		\begin{equation}\label{eq: 2d problem}
		\begin{cases}
			-\Delta_z\psi_y= f&\text{in }A_{\delta',\delta}\\
			\partial_\nu\psi_y=0 &\text{on }\partial B_\delta
		\end{cases}
	\end{equation}	
	for a.e.~ $y\in \Gamma$ and with a right-hand side bounded uniformly in $y$ (and in $\delta$). In this way we reduced the regularity theory for the full problem to the regularity theory for a two-dimensional one \eqref{eq: 2d problem}, where the singularity of the right-hand side actually plays a role. Our claim is that, for a.e.~ $y\in\Gamma$ and $r\in[\delta,\delta'),$
	\begin{equation}\label{eq: claim 2d}
		\left\Vert\psi_y-\fint_{A_{r,\delta}}\psi_y\right\Vert_{L^\infty(A_{r,\delta})}\leq C_\circ r^\gamma
	\end{equation}
	for $C_\circ$ and $\gamma$ independent on both $y$ and $\delta$. Taking the supremum in $y$ the thesis follows, thus we will focus on proving \eqref{eq: claim 2d} knowing \eqref{eq: 2d problem} for a generic right-hand side of class $L^p$. In what follows, we drop the subscript $y$ from $\psi_y$.
	
	\medskip
	
	{\bf Step 2} ($L^2$ bounds). The proof of the claim follows a classical De Giorgi iteration argument. Let $\bar\psi$ be the average of $\psi$ on the annulus. Then, integrating by parts, using H\"older inequality and Sobolev embedding
	\begin{equation*}
		\|\d\psi\|^2_{L^2}\leq \|f(\psi-\bar\psi)\|_{L^1}\leq \|f\|_{L^p}\|\psi-\bar\psi\|_{L^{p'}}\leq \|f\|_{L^p}\|\d\psi\|_{L^2}
	\end{equation*}
	where all the norms are in the annulus. It follows that $\|\d\psi\|_{L^2}\leq \|f\|_{L^p}$. Recall that the Poincar\'e constant $C_P$ on $A_{\delta',\delta}$ can be chosen uniform in $\delta$ due to the geometry of the domain, see for example \cite{AdamsFournier77}. Then, we get
\begin{equation*}
	\int_{A_{\delta',\delta}}|\psi-\bar\psi|^2\leq C_P\int_{A_{\delta',\delta}}|\d \psi|^2\leq C_P\|f\|_{L^p}.
\end{equation*}
Thus the $L^2$ norm of $\psi-\bar\psi$ is bounded, uniformly in $\delta$. Truncating $\psi-\bar\psi$ at a sufficiently large value $\mu>0$ we can make its norm arbitrarily small. Indeed, let $\psi_\mu\coloneqq (\psi-\bar\psi-\mu)_+$; by Chebyshev
		\begin{equation*}
			|\{\psi_\mu>0\}|\leq \frac{1}{\mu^2}\|\psi-\bar\psi\|^2_{L^2(A_{\delta',\delta})}\leq\frac{C}{\mu^2}.
		\end{equation*}
		Then, by using H\"older and Sobolev--Poincaré, 
		\begin{equation}\label{eq: decay L2norm truncation}
			\|\psi_\mu\|_{L^2(A_{\delta',\delta})}\leq \|\psi_\mu\|_{L^{q}(A_{\delta',\delta})}|\{\psi_\mu>0\}|^{1/q'}\leq C\|\d\psi\|_{L^2(A_{\delta',\delta})}|\{\psi_\mu>0\}|^{1/q'}\leq \frac{C}{\mu^{2/q'}} 
		\end{equation}
		for any $q>1$, with a constant independent on $\delta$.

	\medskip
	
	{\bf Step 3} ($L^2$ to $L^\infty$).
	Next we show that the above $L^2$ bounds imply $L^\infty$ bounds smaller annuli. First, note that Caccioppoli's inequality for any subsolution $v$ of 
		\begin{equation*}
			\begin{cases}
				-\Delta v\leq f &\text{in }A_{\delta',\delta}\\
				\partial_\nu v\leq 0 &\text{on }\partial B_\delta.
			\end{cases}
		\end{equation*}
		and any $\varphi\in C^\infty_c(B_{\delta_\circ})$, reads 
		\begin{equation}\label{eq: Caccioppoli}
			\int_{A_{\delta',\delta}} |\d(\varphi v)|^2\leq\|\d\varphi\|^2_\infty\int_{A_{\delta',\delta}} |v|^2+\int_{A_{\delta',\delta}}f\varphi^2v.
		\end{equation}
		Set 
		\begin{equation*}
			A_{k}\coloneqq T_{\delta'(1/2+2^{-k-1})}\setminus T_\delta.
		\end{equation*}
		and observe that $A_0=A_{\delta',\delta}$ and $A_k$ converges to $A_{\delta'/2,\delta}$. By \eqref{eq: decay L2norm truncation}, up to truncating at a sufficiently high level, we can assume that $\psi$ is a subsolution satisfying $\|\psi\|_{L^2(A_{\delta',\delta})}\leq \epsilon \ll 1$ and note that such $\epsilon$ can be chosen independent on $\delta$. Let $\varphi_k$ be radial, smooth cut-off functions, $0\leq \varphi_k\leq 1$, such that $\varphi_k$ is identically $1$ in $A_k$ and vanishes outside of $A_{k-1}$. This implies in particular that $|\d \varphi_k|\leq C2^k$ in $A_{k-1}\setminus A_k$. Set, moreover, 
		\begin{equation*}
			\psi_k\coloneqq (\psi-(1-2^{-k}))_+\quad\text{and}\quad \Psi_k\coloneqq \int (\varphi_k\psi_k)^2.
		\end{equation*}
		Fix $q>1$ large. By H\"older inequality
		\begin{align*}
			\Psi_{k+1}\leq \left(\int(\varphi_{k+1}\psi_{k+1})^{q}\right)^{2/q}|\{\varphi_{k+1}\psi_{k+1}>0\}|^{1-\frac{1}{q}}
		\end{align*}
		and by Chebyshev, we get
		\begin{align}\label{eq: Cheby bound}
			|\{\varphi_{k+1}\psi_{k+1}>0\}|\leq |\{\varphi_{k}\psi_{k}>2^{-k-1}\}|\leq 4^{k+1}\Psi_k.
		\end{align}
		Lastly, by Sobolev and Caccioppoli
		\begin{align*}
			\left(\int(\varphi_{k+1}\psi_{k+1})^{q}\right)^{2/q}&\leq C\int |\d (\varphi_{k+1}\psi_{k+1})|^2\\
			&\leq C4^k\int_{\operatorname{supp}\varphi_{k+1}}|\psi_{k+1}|^2+\int f\varphi_{k+1}^2\psi_{k+1}\\
			&\leq C4^k\Psi_k+\|f\|_p\left(\int(\varphi_{k+1}\psi_{k+1})^{q}\right)^{1/q}|\{\varphi_{k+1}\psi_{k+1}>0\}|^{1-\frac{1}{p}-\frac1q}\\
			&\leq C4^k\Psi_k+\eta\left(\int(\varphi_{k+1}\psi_{k+1})^{q}\right)^{2/q}\\
			&\quad +c(\eta)\|f\|_p^2|\{\varphi_{k+1}\psi_{k+1}>0\}|^{2-\frac{2}{p}-\frac2q}
		\end{align*}
	Thus, up to choosing $\eta$ small enough we can absorb the second term to the left-hand side, and, choosing $q$ large enough, \eqref{eq: Cheby bound} yields the recursive relation
		\begin{equation*}
			\Psi_{k+1}\leq C4^k\Psi_k^{1+\epsilon}.
		\end{equation*}
		Choosing the initial truncation $\Psi_0$ sufficiently small, it follows that $\Psi_k\to 0$ as $k\to \infty$, which implies that $\psi$ is bounded in $A_{\delta'/2,\delta}$.
		
		\medskip
		
		{\bf Step 4} (H\"older decay). After renormalisation, suppose that $\psi\leq 1$ in $A_{r,\delta}$ for some $r<\delta'/2$. We claim that there exists a $\lambda>0$ independent on $\delta$ such that 
		\begin{equation*}
			\psi\leq 1-\lambda\quad\text{in }A_{r/2,\delta}.
		\end{equation*}
		By iteration, this completes the proof. Suppose that $|\{\psi\leq 0\}|\geq |A_{r,\delta}|/2$ (otherwise, we take $-\psi$). Then, define a new sequence 
		\begin{equation*}
			\psi_k\coloneqq 2^k(\psi-(1-2^k))_+.
		\end{equation*}
		Since $0\leq \psi_k\leq 1$, by Caccioppoli inequality 
		\begin{equation*}
			\int_{A_r}|\nabla\psi_k|^2\leq C.
		\end{equation*}
		We apply De Giorgi's isoperimetric inequality (see \cite{FernandezReal-RosOton}) on $\psi_k$, as long as 
		\begin{equation*}
			\int_{A_{r,\delta}}\psi_{k+1}^2\geq \epsilon^2.
		\end{equation*}
		Since 
		\begin{equation*}
			\left|\{\psi_k\geq 1/2\}\cap A_{r,\delta}\right|\geq\left|\{\psi_{k+1}>0\}\cap A_{r,\delta}\right| \geq \int_{A_{r,\delta}}\psi_{k+1}^2\geq \epsilon^2
		\end{equation*}
		we obtain 
		\begin{equation*}
			\left|\{0\leq \psi_k\leq 1/2\}\cap A_{r,\delta}\right|\geq \beta>0.
		\end{equation*}
		Since the sets $\{0\leq \psi_k\leq 1/2\}$ are disjoint, this implies that for some $k_\circ>0$
		\begin{equation*}
			\int_{A_{r,\delta}}\psi_{k_\circ+1}^2< \epsilon^2.
		\end{equation*}
		Then, by Step 3, 
		\begin{equation*}
			\psi_{k_\circ+1}<1/2\quad\text{in }{A_{r/2,\delta}}.
		\end{equation*}
		Scaling back to $\psi$ the conclusion follows.
	\end{proof}
\end{lemma}
\begin{proof}[Proof of Theorem \ref{thm: expansion udelta}] 
	By Proposition \ref{prop: expansion can harm map} we know that the sought expression is valid for $u_\circ$ with some $\beta\in(0,1)$, thus we will just compare the energies. We claim that 
	\begin{equation*}
		\left|\int_{M_\delta}|\d u_\delta|^2-\int_{M_\delta}|\d u_\circ|^2\right|\leq C|\log\delta|^{1/2}\delta^{\gamma/2}.
	\end{equation*}
	for some constants $C,\gamma$ independent on $\delta$. The result then follows, up to redefining $\gamma$. Using H\"older inequality and the identity $|jv|=|\d v|$ for every $v\colon M_\delta\to \s^1$, we have
	\begin{align*}
		\left|\int_{M_\delta}|\d u_\delta|^2-\int_{M_\delta}|\d u_\circ|^2\right|&=\left|\int_{M_\delta}|j u_\delta|^2-|j u_\circ|^2\right|\\
		&=\int_{M_\delta}|ju_\delta+j u_\circ|\cdot|ju_\delta-j u_\circ|\\
		&\leq \sqrt{2}\left(\int_{M_\delta}|ju_\delta|^2+|j u_\circ|^2\right)^{1/2}\left(\int_{M_\delta}|ju_\delta-j u_\circ|^2\right)^{1/2}\\
		&\leq 2\left(\int_{M_\delta}|\d u_\circ|^2\right)^{1/2}\left(\int_{M_\delta}|ju_\delta-j u_\circ|^2\right)^{1/2}
	\end{align*}
	where in the last step we used energy minimality of $u_\delta$. A direct integration using Fermi coordinates and Fubini, along with point (4) of Lemma \ref{lem: propertis Hscr}, yields 
	\begin{equation*}
		\int_{M_\delta}|\d u_\circ|^2\leq C(M,\Gamma)|\log\delta|.
	\end{equation*}
	We have only left to estimate $\|ju_\delta-ju_\circ\|_{L^2(M_\delta)}=\|\d\varphi_\delta\|_{L^2(M_\delta)}$, where $\varphi_\delta$ is the solution to \eqref{eq: minimisation delta with psi} with $u_\alpha=u_\circ$. Recall that $\varphi_\delta$ solves 
	\begin{equation*}
		\begin{cases}
			-\Delta_M\varphi_\delta=0&\text{in }M_\delta\\
			\partial_\nu\varphi_\delta=\langle ju_\circ,\nu\rangle &\text{on }\partial M_\delta.
		\end{cases}
	\end{equation*}
	By point (4) of Lemma \ref{lem: propertis Hscr}, there exists $\lambda\in C^{0,\alpha}$ such that $u_\circ=e^{i\lambda}u_*$, where $u_*(y,z)\coloneqq z/|z|$ in $T_{\delta_\circ}$ and $(y,z)$ are Fermi coordinates. Letting $\psi_\delta\coloneqq \varphi_\delta-\lambda$ we see that 
	\begin{equation}\label{eq: system psi}
		\begin{cases}
			-\Delta_M\psi_\delta=f&\text{in }M_\delta\\
			\partial_\nu\psi_\delta=0 &\text{on }\partial M_\delta.
		\end{cases}
	\end{equation}
	where $f\coloneqq \Delta_M\lambda$. Indeed, $u_*$ is not harmonic due to curvature terms, but its Neumann datum vanishes on $\partial T_\delta$ for every $\delta\in (0,\delta_\circ)$. By Lemma \ref{lem: regularity decay}, there exists a function $\alpha_\delta\in W^{1,\infty}(\Gamma)$ such that $\|\psi_\delta-\alpha_\delta\|_{L^\infty(\partial  T_\delta)}\leq C_\circ \delta^\gamma$ for some constants $C_\circ>0$ and $\gamma\in(0,1)$ independent on $\delta$. Using the $\gamma$-H\"older continuity of $\lambda$ and the corresponding bound $|\partial_\nu\lambda|\vert_{\partial T_\delta}\leq C\delta^{\gamma-1}$,  we see that 
	\begin{align*}
		\int_{M_\delta}|\d\varphi_\delta|^2&=\int_{M_\delta}|\d\psi_\delta+\d\lambda|^2\\
		&\leq \int_{\partial M_\delta}|\psi_\delta+\lambda||\partial_\nu\lambda|\\
		&\leq \int_{\partial M_\delta}|\psi_\delta-\alpha_\delta||\partial_\nu\lambda|+\int_{\partial M_\delta}|\alpha_\delta+\lambda||\partial_\nu\lambda|\\
		&\leq C(\delta^{2\gamma}+\delta^\gamma),
	\end{align*}
	for a constant $C$ independent on $\delta$, which concludes the proof.
\end{proof} 
\subsection{The expansions of $u_p$ and $u_s$}
In this section we prove an expansion similar to the one for $u_\delta$ for the $p$-energy and fractional energy minimisers \eqref{eq: p minim} and \eqref{eq: s minim}. As for $u_\delta$, we consider the minimisers $u_p$ and $u_s$ of 
\begin{equation}\label{eq: ps min circ}
	\min_{u\in W^{1,p}_\circ(M,\s^1)}E_p(u)
\quad\text{and}\quad 
	\min_{u\in H^s_\circ(M,\s^1)}E_s(u)
\end{equation}
respectively, for a fixed $u_\circ\in \mathscr{H}(M,\llb\Gamma\rrb)$.
 The idea is to exploit the fact that $u_p$, $u_s$ and $u_\delta$ are minimisers of ``close'' problems. This principle was already used for $p$-harmonic maps from planar domains to spheres, see \cite{Hardt-Lin1995,Hardt-Lin-Wang1997}.
\begin{theorem}\label{thm: expansion p s}
	Let $u_s,u_p\colon M\to \s^1$ be the minimisers of \eqref{eq: ps min circ}. Then, the energy expansions
	\begin{equation}\label{eq: convergence penergy}
		E_p(u_p)=\frac{2\pi \mathcal{H}^{n-2}(\Gamma)}{2-p}+W_M(\Gamma)+e(u_\circ)+o(1)
	\end{equation}
	and
	\begin{equation}\label{eq: convergence senergy}
		E_s(u_s)=\frac{2\pi \mathcal{H}^{n-2}(\Gamma)}{2-2s}+W_M(\Gamma)+e(u_\circ)+o(1)
	\end{equation}
	hold as $p\to 2^-$ and $s\to 1^-$, respectively.
	\begin{proof}
		We prove first \eqref{eq: convergence penergy}. Given $\epsilon>0$, choose $\delta>0$ such that the $\O(\delta^\beta)$ in the expansion \eqref{eq: expansion udelta} is smaller than $\epsilon$. Choose also $p_*\in (1,2)$ such that for any $p_*<p<2$
		\begin{equation}\label{eq: p estimate 1}
			\left|\frac{1-\delta^{2-p}}{2-p}-\log\delta\right|+\left|\int_{M_\delta}|\d u_p|^p-\int_{M_\delta}|\d u_p|^2\right|<\epsilon.
		\end{equation}  
		We claim that, for $\delta$ small enough,  
		\begin{equation}\label{eq: p estimate 2}
			\int_{T_\delta}|\d u_p|^p\geq \frac{2\pi\delta^{2-p}\mathcal{H}^{n-2}(\Gamma)}{2-p}-\epsilon 
		\end{equation}
		Indeed, the volume element in Fermi coordinates can be expanded as 
		\begin{equation*}
			\sqrt{g(y,z)}=1-H_\Gamma(y)\cdot z+\O(|z|^2)
		\end{equation*}
		where $H_\Gamma$ is the mean curvature vector of $\Gamma$, see \cite{Gray2004} . In particular, one can bound for $|z|<\delta$
		\begin{equation*}
			\sqrt{g(y,z)}\geq 1-c\delta.
		\end{equation*}
		Then, 
		\begin{align*}
			\int_{T_\delta}|\d u_p|^p&\geq \int_\Gamma\int_{B_\delta}|\d u(y,z)|^p\sqrt{g(y,z)}\d V_\Gamma(y)\d z\\
			&\geq (1-c\delta)\int_\Gamma\int_{B_\delta}|\d u(y,z)|^p\d V_\Gamma(y)\d z
		\end{align*}
		By Remark \ref{rem: Fubini} and the degree 1 condition one has, using H\"older inequality
		\begin{equation*}
			2\pi\leq \int_0^{2\pi}|\partial_\vartheta u_p(y,\varrho e^{i\vartheta})|\d\vartheta\leq \left(\int_0^{2\pi}|\partial_\vartheta u_p(y,\varrho e^{i\vartheta})|^p\right)^{1/p}(2\pi)^{1-1/p}
		\end{equation*}
		for a.e. $\varrho\in (0,\delta)$, where $(\varrho,\vartheta)$ are polar coordinates relative to $z\in\C$. This implies
		\begin{equation*}
			\int_0^{2\pi}|\partial_\vartheta u_p(y,\varrho e^{i\vartheta})|^p\geq 2\pi.
		\end{equation*}
		Then, using that $|\d u|^p\geq |\frac1\varrho\partial_\vartheta u|^p$ we find, for $\delta$ sufficiently small,
		\begin{equation}\label{eq: lower bound p energy}
			\begin{split}
				\int_{T_\delta}|\d u_p|^p&\geq (1-c\delta)\int_\Gamma\int_{B_\delta}|\tfrac1\varrho\partial_\vartheta u_p(y,z)|^p\varrho\d V_\Gamma(y)\d \vartheta \d \varrho\\
			&\geq (1-c\delta)2\pi\mathcal{H}^{n-2}(\Gamma)\int_0^\delta \varrho^{1-p}\d \varrho\\
			&\geq \frac{2\pi\delta^{2-p}\mathcal{H}^{n-2}(\Gamma)}{2-p}-\epsilon.
			\end{split}
		\end{equation}
		Combining \eqref{eq: expansion udelta}, \eqref{eq: p estimate 1} and \eqref{eq: p estimate 2}, we get 
		\begin{align*}
			\int_M|\d u_p|^p&=\int_{M_\delta}|\d u_p|^p+\int_{T_\delta}|\d u_p|^p\\
			&\geq \int_{M_\delta}|\d u_p|^2+\int_{T_\delta}|\d u_p|^p-\epsilon \\
			&\geq \int_{M_\delta}|\d u_\delta|^2+\int_{T_\delta}|\d u_p|^p-\epsilon\\
			&\geq 2\pi\mathcal{H}^{n-2}(\Gamma)\log(1/\delta)+W_M(\Gamma)+e(u_\circ)+\int_{T_\delta}|\d u_p|^p-2\epsilon \\
			&\geq \frac{2\pi\mathcal{H}^{n-2}(\Gamma)}{2-p}+W_M(\Gamma)+e(u_\circ)-C\epsilon. 
		\end{align*}
		To find the corresponding upper bound, let $\delta<r<\delta_\circ$ and consider the interpolation map 
		\begin{equation*}
			v\coloneqq 
			\begin{cases}
				u_\delta & \text{in }M_r\\
				\tilde u & \text{in }T_r\setminus T_{r/2}\\
				u_\circ & \text{in }T_{r/2}
			\end{cases}
		\end{equation*}
		where $\tilde u=u_*/|u_*|$ and 
		\begin{equation*}
			u_*(y,z)=\left(\frac{2|z|}{r}-1\right)u_\delta(y,z)+\left(2-\frac{2|z|}{r}\right)u_\circ(y,z)
		\end{equation*}
		for $y\in \Gamma$, $r/2<|z|<r$. Fix again $\epsilon>0$ and note that the same calculation as in \eqref{eq: lower bound p energy} yields, up to choosing $r$ (and $\delta$) sufficiently small, 
		\begin{equation}\label{eq: lower bound 2 energy}
			\int_{T_r\setminus T_\delta}|\d u_\delta|^2\geq 2\pi\mathcal{H}^{n-2}(\Gamma)(\log\tfrac{1}\delta-\log\tfrac{1}r)-\epsilon.
		\end{equation}
		Also, note that the interpolation $\tilde u$ satisfies 
		\begin{equation}\label{eq: lower bound 2 energy in annulus}
			\int_{T_r\setminus T_{r/2}}|\d \tilde u|^2\leq 2\pi\mathcal{H}^{n-2}(\Gamma)\log2+\epsilon.
		\end{equation}
		By \eqref{eq: lower bound 2 energy} and \eqref{eq: lower bound 2 energy in annulus}, we find 
		\begin{equation}\label{eq: second estimate energy p}
			\begin{split}
				\int_{M_{r/2}}|\d v|^2&=\int_{M_\delta}|\d u_\delta|^2-\int_{T_r\setminus T_\delta}|\d u_\delta|^2+\int_{T_r\setminus T_{r/2}}|\d v|^2\\
				&\leq2\pi\mathcal{H}^{n-2}(\Gamma)\log\tfrac{2}r+W_M(\Gamma)+\int_{M}|\omega_\circ|^2+C\epsilon.
			\end{split}
		\end{equation}
		Next, choose again $p_*\in (1,2)$ such that for any $p_*<p<2$
		\begin{equation}\label{eq: p estimate 2 - distance}
			\left|\frac{1-(r/2)^{2-p}}{2-p}-\log\tfrac{2}{r}\right|+\left|\int_{M_{r/2}}|\d v|^p-\int_{M_{r/2}}|\d v|^2\right|<\epsilon
		\end{equation} 
		and note that, by \eqref{eq: second estimate energy p} and \eqref{eq: p estimate 2}, we find 
		\begin{equation*}
			\int_{M_{r/2}}|\d v|^p\leq 2\pi\mathcal{H}^{n-2}(\Gamma)\frac{1-(r/2)^{2-p}}{2-p}+W_M(\Gamma)+e(u_\circ)+C\epsilon.
		\end{equation*}
		By part (4) of Lemma \ref{lem: propertis Hscr} and up to choosing $r$ small enough, we have
		\begin{equation}\label{eq: p estimate ucirc}
			\int_{T_{r/2}}|\d u_\circ|^p\leq 2\pi\mathcal{H}^{n-2}(\Gamma)\frac{(r/2)^{2-p}}{2-p}+\epsilon.
		\end{equation}
		By \eqref{eq: p estimate ucirc} and the minimality of $u_p$, we find the upper bound 
		\begin{equation*}
			\int_{M}|\d u_p|^p\leq \int_{M}|\d v|^p\leq \frac{2\pi\mathcal{H}^{n-2}(\Gamma)}{2-p}+W_M(\Gamma)+e(u_\circ)+C\epsilon.
		\end{equation*}
		which concludes the proof of \eqref{eq: convergence penergy}. 
		
		\medskip
		
		The proof of \eqref{eq: convergence senergy} follows the same lines, but with $1-s$ in place of $2-p$ and  $[u_s]_{H^{s}(M)}^2$ in place of $\int_M|\d u_p|^p$. The upper and lower bounds are obtained analogously after showing the validity of formulae corresponding to \eqref{eq: p estimate 2} and \eqref{eq: p estimate ucirc}, namely 
 	\begin{equation}\label{eq: s estimate 2}
			[u_s]_{H^{s}(T_\delta)}^2\geq \frac{2\pi\delta^{1-s}\mathcal{H}^{n-2}(\Gamma)}{2-2s}-\epsilon 
	\end{equation}
	and 
	\begin{equation}\label{eq: s estimate ucirc}
			[u_\circ]_{H^{s}(T_{r/2})}^2\leq 2\pi\mathcal{H}^{n-2}(\Gamma)\frac{(r/2)^{1-s}}{2-2s}+\epsilon.
		\end{equation}
		To show \eqref{eq: s estimate 2} and \eqref{eq: s estimate ucirc}, we use the following fact. On tubes with a sufficiently small cross-section, the metric becomes ``almost'' diagonal, meaning that the tubular neighbourhood ``almost'' splits as a product space $B_\delta\times \Gamma$. In particular, for every $\epsilon>0$ we can choose $\delta$ small enough such that 
		\begin{equation*}
			\left|[u]_{H^s(T_\delta)}^2-[u]_{H^s(B_\delta\times\Gamma)}^2\right|\leq \epsilon 
		\end{equation*} 
		Then, we use the following equivalent definition of $H^s$ seminorm on an $n$-dimensional manifold $N$ (see \cite{Caselli-Florit-Serra2024})
		\begin{equation*}
			[u]_{H^s(N)}^2=\sum_{k\geq 0}\lambda_k^s\langle u,\phi_k\rangle_{L^2(N)}^2
		\end{equation*}
		where $(\lambda_k,\phi_k)$ are the eigenpairs associated to the Laplace--Beltrami operator on $N$. Now, if $N$ is a product, say $N=N_1\times N_2$, with $(\lambda_k,\phi_k)$ and $(\mu_j,\psi_j)$ being respectively the eigenpairs of $N_1$ and $N_2$, then the eigenpairs of $N$ are given by $(\lambda_k+\mu_j,\phi_k\psi_j)$ and 
		\begin{equation}\label{eq: Hs norm for product}
			[u]_{H^s(N_1\times N_2)}^2=\sum_{k\geq 0}\sum_{j\geq 0}(\lambda_k+\mu_j)^s\left(\int_{N_1\times N_2}u(x,y)\phi_k(x)\psi_j(y)\right)^2.
		\end{equation}
		Using $N_1=B_\delta$, $N_2=\Gamma$ and keeping the same notation for the corresponding eigenpairs, we immediately see that 
		\begin{align*}
			[u_s]_{H^s(B_\delta\times \Gamma)}^2&\geq \sum_{k\geq 0}\lambda_k^s\left(\int_{B_\delta}\left(\frac1{\sqrt{\mathcal{H}^{n-2}(\Gamma)}}\int_{\Gamma}u(x,y)\right)\phi_k(x)\right)^2\\
			&=\mathcal{H}^{n-2}(\Gamma)\left[\overline u_s\right]^2_{H^s(B_\delta)}
		\end{align*}
		where we used that $\mu_0=0$, $\psi_0\equiv \frac1{\sqrt{\mathcal{H}^{n-2}(\Gamma)}}$ and we denoted with
		\begin{equation*}
			\overline{u}_s(z)=\fint_{\Gamma}u_s(z,\cdot)
		\end{equation*}
		the average on $\Gamma$ of $u_s$. By the degree condition on $u_s$, if follows that $\overline{u}_s$ has degree +1 and, by \cite[Lemma 3.6 and Proposition 4.1]{CaselliFregugliaPicenni2024},
		\begin{equation*}
			\left[\overline u_s\right]^2_{H^s(B_\delta)}\geq\frac{2\pi}{2-2s}\delta^{1-s}.
		\end{equation*}
		
	\medskip
	
	To show \eqref{eq: s estimate ucirc}, we approximate once again the tubular neighbourhood $T_\delta$ with a product $B_\delta\times\Gamma$ and use again formula \eqref{eq: Hs norm for product}. Also, by point (4) of Lemma \ref{lem: propertis Hscr} we can approximate $u_\circ$ by the map $u_*(y,z)\coloneqq z/|z|$ arbitrarily well, in the right normal frame. In particular, we have 
	\begin{equation*}
		\left|[u_\circ]_{H^s(T_{r/2})}^2-[u_*]_{H^s(B_{r/2}\times\Gamma)}^2\right|\leq \epsilon 
	\end{equation*}
	Using the $L^2$-orthogonality of $\psi_0$ and $\psi_j$ for every $j\ne 0$, we obtain 
	\begin{align*}
		[u_*]_{H^s(B_{r/2}\times \Gamma)}^2&=\sum_{k\geq 0}\sum_{j\geq 0}(\lambda_k+\mu_j)^s\left(\int_{B_{r/2}\times \Gamma}u_*(x)\phi_k(x)\psi_j(y)\right)^2\\
		&=\mathcal{H}^{n-2}(\Gamma)\sum_{k\geq 0}\lambda_k^s\left(\int_{B_{r/2}}u_*(x)\phi_k(x)\right)^2\\
		&=\mathcal{H}^{n-2}(\Gamma)[u_*]_{H^s(B_{r/2})}^2
	\end{align*}
	where we are slightly abusing notation and using $u_*(z)=z/|z|$ also for the map on the disc.
	Lastly, we use that by \cite[Lemma A.1]{CaselliFregugliaPicenni2024}
	\begin{equation*}
		[u_*]_{H^s(B_{1})}^2=
		\frac{2\pi}{2-2s}+o(1-s)
	\end{equation*}
	Then, using the homogeneity of $u_*$ and scaling, we find that for $s$ sufficiently close to 1,
	\begin{equation*}
		[u_*]_{H^s(B_{r/2})}^2\leq 
		\frac{2\pi}{2-2s}(r/2)^{1-s}+\epsilon.
	\end{equation*}
	\end{proof}
\end{theorem}

\appendix

\section{The renormalised energy in $\R^n$}\label{app: ren en Rn}
In order to get a better picture of the self-interaction energy of the singular set, we give an explicit expression of renormalised energy in the Euclidean space. We start with the simplest case: the interaction energy between two embedded curves $\gamma_1,\gamma_2$ in $\R^3$. The fundamental solution $G(x,y)$ for the Hodge Laplacian on 1-forms is given by
\begin{equation*}
	G(x,y)=\sum_{j,k=1}^3\frac{\delta_{jk}}{4\pi|x-y|}\d x^j\wedge \d y^k
\end{equation*}
where we remark that the sign follows by the fact that the Hodge Laplacian is the \emph{positive} spectrum Laplacian.
Supposing the curves parametrised by arclength, their orientation is given by 
\begin{equation*}
\vec\gamma_k=\sum_{j=1}^3(\dot\gamma_k)_j\d x^j,\quad k=1,2.
\end{equation*}
Then, we have 
\begin{equation*}
	(G(x,y),\vec\gamma_1(x)\wedge\vec\gamma_2(y))_{g_{\R^3}}=\sum_{j,k=1}^3\frac{\delta_{jk}}{4\pi|x-y|}(\dot\gamma_1(x))_j(\dot\gamma_2(y))_k=\frac{\dot\gamma_1(x)\cdot\dot\gamma_2(y)}{4\pi|x-y|}
\end{equation*}
where the dot product is the scalar product in $\R^3$.
Thus, the interaction energy between $\gamma_1$ and $
\gamma_2$ is given by 
\begin{equation*}
	\int_{\gamma_1}\int_{\gamma_2}\frac{\dot\gamma_1(x)\cdot\dot\gamma_2(y)}{4\pi|x-y|}\d\mathcal{H}^{1}(x)\d\mathcal{H}^{1}(y).
\end{equation*}
Observe that the same formula holds if $\gamma_1$ and $\gamma_2$ are different portions of the same connected component, as the desingularisation of Lemma \ref{lem: decomposition w} happens only on the diagonal. 

\medskip

In higher dimension, the interaction is similar. In $\R^n$, the fundamental solution for the Hodge Laplacian is 
\begin{equation*}
	G(x,y)=\sum_{I\in\mathcal{I}_{n-2}}^3\frac{\d x^I\wedge \d y^I}{(n-2)\omega_n|x-y|^{n-2}}.
\end{equation*} 
If $\Sigma_1$ and $\Sigma_2$ are $(n-2)$-dimensional embedded manifolds in $\R^n$ with respective orientations given by 
\begin{equation*}
	\vec{\Sigma}_k(x)=\sum_{J\in \mathcal{I}_{n-2}}\alpha^k_J(x)\d x^J
\end{equation*}
the interaction becomes 
\begin{equation*}
	\int_{\Sigma_1}\int_{\Sigma_2}\frac{
	\langle\vec{\Sigma}_1(x),\vec{\Sigma}_2(y)\rangle}{(n-2)\omega_n|x-y|^{n-2}}\d\mathcal{H}^{n-2}(x)\d\mathcal{H}^{n-2}(y).
\end{equation*}
where
\begin{equation*}
	\langle\vec{\Sigma}_1(x),\vec{\Sigma}_2(y)\rangle=\sum_{J\in \mathcal{I}_{n-2}}\alpha^1_J(x)\alpha^2_J(y)
\end{equation*}
is the pointwise product of $(n-2)$-forms. Observe that the interaction between orthogonal parts of the submanifold vanishes, while it is maximised on parallel pieces, with a sign depending on the orientation.

\subsection{Renormalised energy and magnetic inductances}
The renormalised energy appears to be also deeply related to the theory of electromagnetism; it corresponds to the energy stored in the magnetic field generated by currents running along wires. Precisely, let $\gamma_1,\dots,\gamma_m$ be $m$ closed wires embedded in $\R^3$ in which currents $I_1,\dots, I_m$ are flowing parallel to the orientations $\vec\gamma_1,\dots,\vec\gamma_m$. The total electric current of the system is given by 
\begin{equation*}
	J=\sum_{j=1}^mI_j\llbracket\gamma_j\rrbracket.
\end{equation*}
The magnetic field $B$ generated by $J$ obeys Maxwell's equations 
\begin{equation*}
	\nabla\times B=J,\quad \nabla\cdot B=0
\end{equation*}
or, in the language of forms (recalling that the magnetic field is identified with a 2-form $B\in\Omega^2(\R^n)$,
\begin{equation*}
	\d^*B=J,\quad \d  B  =0.
\end{equation*}
The second equation implies, by Poincar\'e Lemma, that $B=\d A$ for some 1-form $A$ and by the first equation
\begin{equation*}
	\d^*\d A=J.
\end{equation*}
Exactly as in \eqref{eq: d*A harmonic} we can show that $\d^*A$ is harmonic, thus $\d\d^*A=0$. This means that 
\begin{equation*}
	\Delta A=J
\end{equation*}
which is the same equation satisfied by the $A$ in \S\ref{subs: ren en}, up to modifying the definition to include the current intensities. The magnetic energy is given by
\begin{equation*}
	\int|B|^2=\int|\d A|^2
\end{equation*}
which is exactly the quantity computed in \eqref{eq: square norm singular term}. Thus, integrating in the complement of a $\delta$-tubular neighbourhood of the wires and after a $|\log\delta|$ renormalisation proportional to the length of the wires, we find the well known formula for the inductances
\begin{align*}
	W_{\R^3}(\gamma_1,\dots,\gamma_m)&=\sum_{j=1}^mI_j^2\int_{\gamma_j}\langle R_j(y),\vec\gamma_j(y)\rangle \d\mathcal{H}^{n-2}(y)\\
	&+\sum_{j\ne k}I_jI_k\int_{\gamma_j}\int_{\gamma_k}\frac{ \vec\gamma_j(x)\cdot \vec\gamma_k(y)}{|x-y|}\d\mathcal{H}^{n-2}(x)\d\mathcal{H}^{n-2}(y)
\end{align*}
where the first term is the sum of the (renormalised) self-inductances and the second term the mutual inductances of the wires, see \cite{Jackson1975}.

\bibliography{Bibliography} 

\providecommand{\bysame}{\leavevmode\hbox to3em{\hrulefill}\thinspace}
\providecommand{\noopsort}[1]{}
\providecommand{\mr}[1]{\href{http://www.ams.org/mathscinet-getitem?mr=#1}{MR~#1}}
\providecommand{\zbl}[1]{\href{http://www.zentralblatt-math.org/zmath/en/search/?q=an:#1}{Zbl~#1}}
\providecommand{\jfm}[1]{\href{http://www.emis.de/cgi-bin/JFM-item?#1}{JFM~#1}}
\providecommand{\arxiv}[1]{\href{http://www.arxiv.org/abs/#1}{arXiv~#1}}
\providecommand{\doi}[1]{\url{https://doi.org/#1}}
\providecommand{\MR}{\relax\ifhmode\unskip\space\fi MR }
\providecommand{\MRhref}[2]{%
  \href{http://www.ams.org/mathscinet-getitem?mr=#1}{#2}
}
\providecommand{\href}[2]{#2}
\begin{thebibliography}{LMWW21}

\bibitem[AF77]{AdamsFournier77}
\bgroup\scshape{}R.~A. Adams\egroup{} and \bgroup\scshape{}J.~Fournier\egroup{}, Cone conditions and properties of {S}obolev spaces,  \emph{J. Math. Anal. Appl.} \textbf{61} no.~3 (1977), 713--734. \mr{463902}.  \doi{10.1016/0022-247X(77)90173-1}.

\bibitem[ABO03]{Alberti-Baldo-Orlandi2003}
\bgroup\scshape{}G.~Alberti\egroup{}, \bgroup\scshape{}S.~Baldo\egroup{}, and \bgroup\scshape{}G.~Orlandi\egroup{}, Functions with prescribed singularities,  \emph{J. Eur. Math. Soc. (JEMS)} \textbf{5} no.~3 (2003), 275--311. \mr{2002215}.  \doi{10.1007/s10097-003-0053-5}.

\bibitem[ABO05]{Alberti-Baldo-Orlandi2005}
\bgroup\scshape{}G.~Alberti\egroup{}, \bgroup\scshape{}S.~Baldo\egroup{}, and \bgroup\scshape{}G.~Orlandi\egroup{}, Variational convergence for functionals of {G}inzburg-{L}andau type,  \emph{Indiana Univ. Math. J.} \textbf{54} no.~5 (2005), 1411--1472. \mr{2177107}.  \doi{10.1512/iumj.2005.54.2601}.

\bibitem[BD23]{Badran-DelPino2023}
\bgroup\scshape{}M.~Badran\egroup{} and \bgroup\scshape{}M.~{Del Pino}\egroup{}, Entire solutions to 4 dimensional {G}inzburg-{L}andau equations and codimension 2 minimal submanifolds,  \emph{Adv. Math.} \textbf{435} no.~part A (2023), Paper No. 109365, 73. \mr{4658832}.  \doi{10.1016/j.aim.2023.109365}.

\bibitem[BD24]{Badran-DelPino2024}
\bgroup\scshape{}M.~Badran\egroup{} and \bgroup\scshape{}M.~{Del Pino}\egroup{}, Solutions of the {G}inzburg-{L}andau equations concentrating on codimension-2 minimal submanifolds,  \emph{J. Lond. Math. Soc. (2)} \textbf{109} no.~1 (2024), Paper No. e12851, 31. \mr{4754434}.  \doi{10.1112/jlms.12851}.

\bibitem[BBH93]{Bethuel-Brezis-Helein1993}
\bgroup\scshape{}F.~Bethuel\egroup{}, \bgroup\scshape{}H.~Brezis\egroup{}, and \bgroup\scshape{}F.~H\'{e}lein\egroup{}, Asymptotics for the minimization of a {G}inzburg-{L}andau functional,  \emph{Calc. Var. Partial Differential Equations} \textbf{1} no.~2 (1993), 123--148. \mr{1261720}.  \doi{10.1007/BF01191614}.

\bibitem[BBH94]{Bethuel-Brezis-Helein1994}
\bgroup\scshape{}F.~Bethuel\egroup{}, \bgroup\scshape{}H.~Brezis\egroup{}, and \bgroup\scshape{}F.~H\'{e}lein\egroup{}, \emph{Ginzburg-{L}andau vortices}, \emph{Progress in Nonlinear Differential Equations and their Applications} \textbf{13}, Birkh\"{a}user Boston, Inc., Boston, MA, 1994. \mr{1269538}.  \doi{10.1007/978-1-4612-0287-5}.

\bibitem[BM21]{Brezis-Mironescu2021}
\bgroup\scshape{}H.~Brezis\egroup{} and \bgroup\scshape{}P.~Mironescu\egroup{}, \emph{Sobolev maps to the circle---from the perspective of analysis, geometry, and topology}, \emph{Progress in Nonlinear Differential Equations and their Applications} \textbf{96}, Birkh\"{a}user/Springer, New York, [2021] \copyright 2021. \mr{4390036}.  \doi{10.1007/978-1-0716-1512-6}.

\bibitem[CRS10]{Caffarelli-Roquejoffre-Savin2010}
\bgroup\scshape{}L.~Caffarelli\egroup{}, \bgroup\scshape{}J.-M. Roquejoffre\egroup{}, and \bgroup\scshape{}O.~Savin\egroup{}, Nonlocal minimal surfaces,  \emph{Comm. Pure Appl. Math.} \textbf{63} no.~9 (2010), 1111--1144. \mr{2675483}.  \doi{10.1002/cpa.20331}.

\bibitem[CFS24a]{Caselli-Florit-Serra2024}
\bgroup\scshape{}M.~Caselli\egroup{}, \bgroup\scshape{}E.~{Florit-Simon}\egroup{}, and \bgroup\scshape{}J.~Serra\egroup{}, Fractional {S}obolev spaces on {R}iemannian manifolds, 2024. Available at \url{https://arxiv.org/abs/2402.04076}.

\bibitem[CFS24b]{Caselli-Florit-Serra2024Yau}
\bgroup\scshape{}M.~Caselli\egroup{}, \bgroup\scshape{}E.~{Florit-Simon}\egroup{}, and \bgroup\scshape{}J.~Serra\egroup{}, Yau's conjecture for nonlocal minimal surfaces, 2024. Available at \url{https://arxiv.org/abs/2306.07100}.

\bibitem[CFP24]{CaselliFregugliaPicenni2024}
\bgroup\scshape{}M.~Caselli\egroup{}, \bgroup\scshape{}M.~Freguglia\egroup{}, and \bgroup\scshape{}N.~Picenni\egroup{}, A nonlocal approximation of the area in codimension two, 2024. Available at \url{https://arxiv.org/abs/2406.13696}.

\bibitem[CDSV23]{Chan-Dipierro-Serra-Valdinoci2023}
\bgroup\scshape{}H.~Chan\egroup{}, \bgroup\scshape{}S.~Dipierro\egroup{}, \bgroup\scshape{}J.~Serra\egroup{}, and \bgroup\scshape{}E.~Valdinoci\egroup{}, Nonlocal approximation of minimal surfaces: optimal estimates from stability, 2023. Available at \url{https://arxiv.org/abs/2308.06328}.

\bibitem[CM20]{Chodosh-Mantoulidis2020}
\bgroup\scshape{}O.~Chodosh\egroup{} and \bgroup\scshape{}C.~Mantoulidis\egroup{}, Minimal surfaces and the {A}llen-{C}ahn equation on 3-manifolds: index, multiplicity, and curvature estimates,  \emph{Ann. of Math. (2)} \textbf{191} no.~1 (2020), 213--328. \mr{4045964}.  \doi{10.4007/annals.2020.191.1.4}.

\bibitem[CJS21]{Colinet-Jerrard-Sternberg2021}
\bgroup\scshape{}A.~Colinet\egroup{}, \bgroup\scshape{}R.~L. Jerrard\egroup{}, and \bgroup\scshape{}P.~Sternberg\egroup{}, Solutions of the {G}inzburg-{L}andau equations with vorticity concentrating near a nondegenerate geodesic, 2021. Available at \url{https://arxiv.org/abs/2101.03575}.

\bibitem[CJ17]{ContrerasJerrard2017}
\bgroup\scshape{}A.~Contreras\egroup{} and \bgroup\scshape{}R.~L. Jerrard\egroup{}, Nearly parallel vortex filaments in the 3{D} {G}inzburg-{L}andau equations,  \emph{Geom. Funct. Anal.} \textbf{27} no.~5 (2017), 1161--1230. \mr{3714719}.  \doi{10.1007/s00039-017-0425-8}.

\bibitem[DR11]{DaLio-Riviere2011}
\bgroup\scshape{}F.~{Da Lio}\egroup{} and \bgroup\scshape{}T.~Rivi\`ere\egroup{}, Three-term commutator estimates and the regularity of {$\frac12$}-harmonic maps into spheres,  \emph{Anal. PDE} \textbf{4} no.~1 (2011), 149--190. \mr{2783309}.  \doi{10.2140/apde.2011.4.149}.

\bibitem[DDMR22]{Davila-delPino-Medina-Rodiac2022}
\bgroup\scshape{}J.~D\'{a}vila\egroup{}, \bgroup\scshape{}M.~{Del Pino}\egroup{}, \bgroup\scshape{}M.~Medina\egroup{}, and \bgroup\scshape{}R.~Rodiac\egroup{}, Interacting helical vortex filaments in the three-dimensional {G}inzburg-{L}andau equation,  \emph{J. Eur. Math. Soc. (JEMS)} \textbf{24} no.~12 (2022), 4143--4199. \mr{4493621}.  \doi{10.4171/jems/1175}.

\bibitem[DDW18]{Davila-delPino-Wei2018}
\bgroup\scshape{}J.~D\'{a}vila\egroup{}, \bgroup\scshape{}M.~{Del Pino}\egroup{}, and \bgroup\scshape{}J.~Wei\egroup{}, Nonlocal {$s$}-minimal surfaces and {L}awson cones,  \emph{J. Differential Geom.} \textbf{109} no.~1 (2018), 111--175. \mr{3798717}.  \doi{10.4310/jdg/1525399218}.

\bibitem[DP24]{DePhilippis-Pigati2024}
\bgroup\scshape{}G.~{De Philippis}\egroup{} and \bgroup\scshape{}A.~Pigati\egroup{}, Non-degenerate minimal submanifolds as energy concentration sets: a variational approach,  \emph{Comm. Pure Appl. Math.} \textbf{77} no.~8 (2024), 3581--3627. \mr{4764749}.

\bibitem[DK08]{DelPinoKowalczyk2008}
\bgroup\scshape{}M.~{Del Pino}\egroup{} and \bgroup\scshape{}M.~Kowalczyk\egroup{}, Renormalized energy of interacting {G}inzburg-{L}andau vortex filaments,  \emph{J. Lond. Math. Soc. (2)} \textbf{77} no.~3 (2008), 647--665. \mr{2418297}.  \doi{10.1112/jlms/jdm126}.

\bibitem[FR22]{FernandezReal-RosOton}
\bgroup\scshape{}X.~{Fern\'{a}ndez-Real}\egroup{} and \bgroup\scshape{}X.~{Ros-Oton}\egroup{}, \emph{Regularity theory for elliptic {PDE}}, \emph{Zurich Lectures in Advanced Mathematics} \textbf{28}, EMS Press, Berlin, [2022] \copyright 2022. \mr{4560756}.  \doi{10.4171/zlam/28}.

\bibitem[{Flo}24]{Florit2024}
\bgroup\scshape{}E.~{Florit-Simon}\egroup{}, Weyl {L}aw and convergence in the classical limit for min-max nonlocal minimal surfaces, 2024. Available at \url{https://arxiv.org/abs/2406.12162}.

\bibitem[GG18]{Gaspar-Guaraco2018}
\bgroup\scshape{}P.~Gaspar\egroup{} and \bgroup\scshape{}M.~A.~M. Guaraco\egroup{}, The {A}llen-{C}ahn equation on closed manifolds,  \emph{Calc. Var. Partial Differential Equations} \textbf{57} no.~4 (2018), Paper No. 101, 42. \mr{3814054}.  \doi{10.1007/s00526-018-1379-x}.

\bibitem[Gra04]{Gray2004}
\bgroup\scshape{}A.~Gray\egroup{}, \emph{Tubes}, second ed., \emph{Progress in Mathematics} \textbf{221}, Birkh\"{a}user Verlag, Basel, 2004, With a preface by Vicente Miquel. \mr{2024928}.  \doi{10.1007/978-3-0348-7966-8}.

\bibitem[GS16]{Grey-Sinnamon2016}
\bgroup\scshape{}W.~Grey\egroup{} and \bgroup\scshape{}G.~Sinnamon\egroup{}, The inclusion problem for mixed norm spaces,  \emph{Trans. Amer. Math. Soc.} \textbf{368} no.~12 (2016), 8715--8736. \mr{3551586}.  \doi{10.1090/tran6665}.

\bibitem[Gua18]{Guaraco2018}
\bgroup\scshape{}M.~A.~M. Guaraco\egroup{}, Min-max for phase transitions and the existence of embedded minimal hypersurfaces,  \emph{J. Differential Geom.} \textbf{108} no.~1 (2018), 91--133. \mr{3743704}.  \doi{10.4310/jdg/1513998031}.

\bibitem[HL87]{Hardt-Lin1987}
\bgroup\scshape{}R.~Hardt\egroup{} and \bgroup\scshape{}F.-H. Lin\egroup{}, Mappings minimizing the {$L^p$} norm of the gradient,  \emph{Comm. Pure Appl. Math.} \textbf{40} no.~5 (1987), 555--588. \mr{896767}.  \doi{10.1002/cpa.3160400503}.

\bibitem[HL95]{Hardt-Lin1995}
\bgroup\scshape{}R.~Hardt\egroup{} and \bgroup\scshape{}F.~Lin\egroup{}, Singularities for {$p$}-energy minimizing unit vectorfields on planar domains,  \emph{Calc. Var. Partial Differential Equations} \textbf{3} no.~3 (1995), 311--341. \mr{1385291}.  \doi{10.1007/BF01189395}.

\bibitem[HLW97]{Hardt-Lin-Wang1997}
\bgroup\scshape{}R.~Hardt\egroup{}, \bgroup\scshape{}F.~Lin\egroup{}, and \bgroup\scshape{}C.~Wang\egroup{}, Singularities of {$p$}-energy minimizing maps,  \emph{Comm. Pure Appl. Math.} \textbf{50} no.~5 (1997), 399--447. \mr{1443054}.  \doi{10.1002/(SICI)1097-0312(199705)50:5<399::AID-CPA1>3.0.CO;2-4}.

\bibitem[Hat02]{Hatcher2002}
\bgroup\scshape{}A.~Hatcher\egroup{}, \emph{Algebraic topology}, Cambridge University Press, Cambridge, 2002. \mr{1867354}.

\bibitem[HT00]{Hutchinson-Tonegawa2000}
\bgroup\scshape{}J.~E. Hutchinson\egroup{} and \bgroup\scshape{}Y.~Tonegawa\egroup{}, Convergence of phase interfaces in the van der {W}aals-{C}ahn-{H}illiard theory,  \emph{Calc. Var. Partial Differential Equations} \textbf{10} no.~1 (2000), 49--84. \mr{1803974}.  \doi{10.1007/PL00013453}.

\bibitem[IJ21]{Ignat-Jerrard2021}
\bgroup\scshape{}R.~Ignat\egroup{} and \bgroup\scshape{}R.~L. Jerrard\egroup{}, Renormalized energy between vortices in some {G}inzburg-{L}andau models on 2-dimensional {R}iemannian manifolds,  \emph{Arch. Ration. Mech. Anal.} \textbf{239} no.~3 (2021), 1577--1666. \mr{4215198}.  \doi{10.1007/s00205-020-01598-0}.

\bibitem[Jac75]{Jackson1975}
\bgroup\scshape{}J.~D. Jackson\egroup{}, \emph{Classical electrodynamics}, second ed., John Wiley \& Sons, Inc., New York-London-Sydney, 1975. \mr{436782}.

\bibitem[JS02a]{Jerrard-Soner2002}
\bgroup\scshape{}R.~L. Jerrard\egroup{} and \bgroup\scshape{}H.~M. Soner\egroup{}, Functions of bounded higher variation,  \emph{Indiana Univ. Math. J.} \textbf{51} no.~3 (2002), 645--677. \mr{1911049}.  \doi{10.1512/iumj.2002.51.2229}.

\bibitem[JS02b]{Jerrard-Soner2002CalcVar}
\bgroup\scshape{}R.~L. Jerrard\egroup{} and \bgroup\scshape{}H.~M. Soner\egroup{}, The {J}acobian and the {G}inzburg-{L}andau energy,  \emph{Calc. Var. Partial Differential Equations} \textbf{14} no.~2 (2002), 151--191. \mr{1890398}.  \doi{10.1007/s005260100093}.

\bibitem[JS02c]{Jerrard-Soner2002JFA}
\bgroup\scshape{}R.~L. Jerrard\egroup{} and \bgroup\scshape{}H.~M. Soner\egroup{}, Limiting behavior of the {G}inzburg-{L}andau functional,  \emph{J. Funct. Anal.} \textbf{192} no.~2 (2002), 524--561. \mr{1923413}.  \doi{10.1006/jfan.2001.3906}.

\bibitem[LMWW21]{Liu-Ma-Wei-Wu2021}
\bgroup\scshape{}Y.~Liu\egroup{}, \bgroup\scshape{}X.~Ma\egroup{}, \bgroup\scshape{}J.~Wei\egroup{}, and \bgroup\scshape{}W.~Wu\egroup{}, Entire solutions of the magnetic {G}inzburg-{L}andau equation in $\mathbb{R}^4$, 2021. Available at \url{https://arxiv.org/abs/2108.02754}.

\bibitem[LWY24]{Liu-Wei-Ye2024}
\bgroup\scshape{}Y.~Liu\egroup{}, \bgroup\scshape{}J.~Wei\egroup{}, and \bgroup\scshape{}Z.~Ye\egroup{}, Stable solutions of $u(1)$ yang-mills-higgs model in $\mathbb{R}^4$, 2024. Available at \url{https://arxiv.org/abs/2411.05447}.

\bibitem[{Mar}23]{MarxKuo2023}
\bgroup\scshape{}J.~{Marx-Kuo}\egroup{}, Second inner variations of energy and index of codimension $2$ minimal submanifolds, 2023. Available at \url{https://arxiv.org/abs/2307.09733}.

\bibitem[Mes09]{Mesaric2009}
\bgroup\scshape{}J.~Mesaric\egroup{}, \emph{Existence of critical points for the {G}inzburg-{L}andau functional on {R}iemannian manifolds}, ProQuest LLC, Ann Arbor, MI, 2009, Thesis (Ph.D.)--University of Toronto (Canada). \mr{2753145}.  Available at \url{http://gateway.proquest.com/openurl?url_ver=Z39.88-2004&rft_val_fmt=info:ofi/fmt:kev:mtx:dissertation&res_dat=xri:pqdiss&rft_dat=xri:pqdiss:NR61033}.

\bibitem[MRV21]{Monteil-Rodiac-VanSchaftingen2021}
\bgroup\scshape{}A.~Monteil\egroup{}, \bgroup\scshape{}R.~Rodiac\egroup{}, and \bgroup\scshape{}J.~{Van Schaftingen}\egroup{}, Ginzburg-{L}andau relaxation for harmonic maps on planar domains into a general compact vacuum manifold,  \emph{Arch. Ration. Mech. Anal.} \textbf{242} no.~2 (2021), 875--935. \mr{4331018}.  \doi{10.1007/s00205-021-01695-8}.

\bibitem[MRV22]{Monteil-Rodiac-VanSchaftingen2022}
\bgroup\scshape{}A.~Monteil\egroup{}, \bgroup\scshape{}R.~Rodiac\egroup{}, and \bgroup\scshape{}J.~{Van Schaftingen}\egroup{}, Renormalised energies and renormalisable singular harmonic maps into a compact manifold on planar domains,  \emph{Math. Ann.} \textbf{383} no.~3-4 (2022), 1061--1125. \mr{4458397}.  \doi{10.1007/s00208-021-02204-8}.

\bibitem[PPS24]{Parise-Pigati-Stern2024}
\bgroup\scshape{}D.~Parise\egroup{}, \bgroup\scshape{}A.~Pigati\egroup{}, and \bgroup\scshape{}D.~Stern\egroup{}, Convergence of the self-dual {$U(1)$}-{Y}ang-{M}ills-{H}iggs energies to the {$(n-2)$}-area functional,  \emph{Comm. Pure Appl. Math.} \textbf{77} no.~1 (2024), 670--730. \mr{4666634}.  \doi{10.1002/cpa.22150}.

\bibitem[Pet06]{Petersen2006}
\bgroup\scshape{}P.~Petersen\egroup{}, \emph{Riemannian geometry}, second ed., \emph{Graduate Texts in Mathematics} \textbf{171}, Springer, New York, 2006. \mr{2243772}.

\bibitem[PS21]{Pigati-Stern2021}
\bgroup\scshape{}A.~Pigati\egroup{} and \bgroup\scshape{}D.~Stern\egroup{}, Minimal submanifolds from the abelian {H}iggs model,  \emph{Invent. Math.} \textbf{223} no.~3 (2021), 1027--1095. \mr{4213771}.  \doi{10.1007/s00222-020-01000-6}.

\bibitem[PS23]{Pigati-Stern2023}
\bgroup\scshape{}A.~Pigati\egroup{} and \bgroup\scshape{}D.~Stern\egroup{}, Quantization and non-quantization of energy for higher-dimensional {G}inzburg-{L}andau vortices,  \emph{Ars Inven. Anal.} (2023), Paper No. 3, 55. \mr{4603940}.

\bibitem[Riv96]{Riviere1996}
\bgroup\scshape{}T.~Rivi\`ere\egroup{}, Line vortices in the {${\rm U}(1)$}-{H}iggs model,  \emph{ESAIM Contr\^{o}le Optim. Calc. Var.} \textbf{1} (1995/96), 77--167. \mr{1394302}.  \doi{10.1051/cocv:1996103}.

\bibitem[Sch95]{Schwarz1995}
\bgroup\scshape{}G.~Schwarz\egroup{}, \emph{Hodge decomposition---a method for solving boundary value problems}, \emph{Lecture Notes in Mathematics} \textbf{1607}, Springer-Verlag, Berlin, 1995. \mr{1367287}.  \doi{10.1007/BFb0095978}.

\bibitem[Sco95]{Scott1995}
\bgroup\scshape{}C.~Scott\egroup{}, {$L^p$} theory of differential forms on manifolds,  \emph{Trans. Amer. Math. Soc.} \textbf{347} no.~6 (1995), 2075--2096. \mr{1297538}.  \doi{10.2307/2154923}.

\bibitem[Ser24]{Serra2024}
\bgroup\scshape{}J.~Serra\egroup{}, Nonlocal minimal surfaces: recent developments, applications, and future directions,  \emph{SeMA J.} \textbf{81} no.~2 (2024), 165--191. \mr{4743530}.  \doi{10.1007/s40324-023-00345-1}.

\bibitem[Sim83]{Simon1983}
\bgroup\scshape{}L.~Simon\egroup{}, \emph{Lectures on geometric measure theory}, \emph{Proceedings of the Centre for Mathematical Analysis, Australian National University} \textbf{3}, Australian National University, Centre for Mathematical Analysis, Canberra, 1983. \mr{756417}.

\bibitem[ST04]{Stefanov-Torres2004}
\bgroup\scshape{}A.~Stefanov\egroup{} and \bgroup\scshape{}R.~H. Torres\egroup{}, Calder\'{o}n-{Z}ygmund operators on mixed {L}ebesgue spaces and applications to null forms,  \emph{J. London Math. Soc. (2)} \textbf{70} no.~2 (2004), 447--462. \mr{2078904}.  \doi{10.1112/S0024610704005502}.

\bibitem[Ste20]{Stern2020}
\bgroup\scshape{}D.~Stern\egroup{}, {$p$}-harmonic maps to {$S^1$} and stationary varifolds of codimension two,  \emph{Calc. Var. Partial Differential Equations} \textbf{59} no.~6 (2020), Paper No. 187, 47. \mr{4160864}.  \doi{10.1007/s00526-020-01859-6}.

\bibitem[Ste21]{Stern2021}
\bgroup\scshape{}D.~Stern\egroup{}, Existence and limiting behavior of min-max solutions of the {G}inzburg-{L}andau equations on compact manifolds,  \emph{J. Differential Geom.} \textbf{118} no.~2 (2021), 335--371. \mr{4278697}.  \doi{10.4310/jdg/1622743143}.

\bibitem[Str94]{Struwe1994}
\bgroup\scshape{}M.~Struwe\egroup{}, On the asymptotic behavior of minimizers of the {G}inzburg-{L}andau model in {$2$} dimensions,  \emph{Differential Integral Equations} \textbf{7} no.~5-6 (1994), 1613--1624. \mr{1269674}.

\bibitem[Ton05]{Tonegawa2005}
\bgroup\scshape{}Y.~Tonegawa\egroup{}, On stable critical points for a singular perturbation problem,  \emph{Comm. Anal. Geom.} \textbf{13} no.~2 (2005), 439--459. \mr{2154826}.  Available at \url{http://projecteuclid.org/euclid.cag/1117656992}.

\bibitem[TW12]{Tonegawa-Wickramasekera2012}
\bgroup\scshape{}Y.~Tonegawa\egroup{} and \bgroup\scshape{}N.~Wickramasekera\egroup{}, Stable phase interfaces in the van der {W}aals--{C}ahn--{H}illiard theory,  \emph{J. Reine Angew. Math.} \textbf{668} (2012), 191--210. \mr{2948876}.  \doi{10.1515/crelle.2011.134}.

\bibitem[VV23]{VanSchaftingen-VanVaerenbergh2023}
\bgroup\scshape{}J.~{Van Schaftingen}\egroup{} and \bgroup\scshape{}B.~{Van Vaerenbergh}\egroup{}, Asymptotic behavior of minimizing {$p$}-harmonic maps when {$p \nearrow 2$} in dimension 2,  \emph{Calc. Var. Partial Differential Equations} \textbf{62} no.~8 (2023), Paper No. 229, 45. \mr{4642640}.  \doi{10.1007/s00526-023-02568-6}.

\bibitem[War83]{Warner1983}
\bgroup\scshape{}F.~W. Warner\egroup{}, \emph{Foundations of differentiable manifolds and {L}ie groups}, \emph{Graduate Texts in Mathematics} \textbf{94}, Springer-Verlag, New York-Berlin, 1983, Corrected reprint of the 1971 edition. \mr{722297}.

\end{thebibliography}
\bibliographystyle{aomalpha}

\end{document}